\documentclass[11.5pt]{amsart}
\usepackage{amsmath,amssymb,color}
\usepackage{amsthm}
\usepackage{hyperref}
\usepackage[margin=1.0in]{geometry}
\usepackage{cleveref}
\usepackage[cm]{fullpage}
\usepackage{enumitem}
\usepackage{amscd}
\usepackage{systeme,mathtools}
\usepackage{amsfonts}
\usepackage{amsthm}
\usepackage{tikz-cd}
\usepackage{stmaryrd}
\usepackage{lipsum}
\usepackage{appendix}
\usepackage{hyperref}
\usepackage{cleveref}
\DeclareMathOperator{\divg}{div}

\DeclareMathOperator{\Rc}{Rc}

\newcommand{\ddt}[1]{\frac{\partial #1}{\partial t}}

\newcommand{\ov}[1]{\overline{#1}}

\begin{document}

\theoremstyle{definition}
\newtheorem{claim}{Claim}
\theoremstyle{plain}
\newtheorem{proposition}{Proposition}[section]
\newtheorem{theorem}[proposition]{Theorem}
\newtheorem{lemma}[proposition]{Lemma}
\newtheorem{corollary}[proposition]{Corollary}
\theoremstyle{definition}
\newtheorem{defn}[proposition]{Definition}
\theoremstyle{remark}
\newtheorem{remark}[proposition]{Remark}
\theoremstyle{definition}
\newtheorem{example}[proposition]{Example}
\theoremstyle{definition}
\newtheorem*{Motivation}{Motivation}

\newcommand{\Addresses}{{
  \footnotesize
 
  \par\nopagebreak
  \textsc{Department of Mathematics, McGill University, Montreal, Canada }\par\nopagebreak
  \textit{E-mail address}: \texttt{kuan-hui.lee@mcgill.ca}

}}

\title{Stability of Hyperkähler Flow}

\author{KUAN-HUI LEE}
\date{}

\begin{abstract}
In this work, we discuss the stability of Donaldson's flow of surfaces in a hyperkähler 4-manifold. In \cite{WT2}, Wang and Tsai proved a uniqueness theorem and $C^1$ dynamic stability theorem of the mean curvature flow for minimal surface. We extend their results and obtain a similar dynamic stability thoerem of the hyperkähler flow.
\end{abstract}

\maketitle
\Addresses
\section{Introduction}

Mean curvature flow is a natural evolution equation in extrinsic geometry and shares many features with Hamilton’s Ricci flow \cite{CK,10.4310/jdg/1214436922} from intrinsic geometry. Mean curvature flow and its variants have striking applications in geometry, topology, and general relativity. In the codimension-one case, Huisken obtained several convergence results \cite{H1,H2}. For higher codimension, the analysis becomes more difficult. Although some convergence results have been established \cite{W1,W2,W3}, they often require additional geometric structures to proceed. One particularly interesting setting is the Lagrangian condition. In \cite{S}, it was shown that the Lagrangian condition is preserved under the mean curvature flow, and several convergence results were also obtained \cite{10.1007/978-3-642-22842-1_9}. Motivated by the mean curvature flow, in this paper we study a different type of geometric flow, which we introduce below.

 In 1999, Donaldson \cite{D} used the moment map and diffeomorphism to construct some geometric evolution equations. In particular, one of the geometric evolution equation in hyperk\"{a}hler 4-manifolds case is similar to the mean curvature flow. Let $S$ be a Riemann surface with volume form $\rho$ and $(M,\overline{g},I,J,K)$ be a hyperk\"{a}hler 4-manifold with three K\"{a}hler forms $\overline{\omega}_1,\overline{\omega_2},\overline{\omega}_3$. The Donaldson's flow in hyperk\"{a}hler 4-manifolds (called $H$-flow) is given by
\begin{align*}           
    \frac{\partial f}{\partial t}=If_\star(\xi_1)+Jf_\star(\xi_2)+Kf_\star(\xi_3), \quad \textrm{ $f:S\to M$ is an immersion,}
\end{align*} 
 where $\xi_i$ is the Hamiltonian vector field on $S$ with respect to $\frac{f^\star(\overline{\omega}_i)}{\rho}$, $i=1,2,3$.
 
 Donaldson \cite{D}, Song and Weinkove \cite{SW} found that if we define $\lambda=\frac{d\mu}{\rho}$, where $d\mu$ is the induced volume form. Then $H$-flow can be written as 
 \begin{align*}
     \frac{\partial f}{\partial t}=\lambda\nabla\lambda+\lambda^2 H,
 \end{align*}
 where $H$ is the mean curvature vector. Besides, the $H-$flow can be viewed as the gradient flow of the hyperk\"{a}ler energy
 \begin{align}
     E(f)=\int_S \lambda^2 \rho. \label{E}
 \end{align}
 
Although the H-flow resembles the mean curvature flow, it does not generally satisfy the same long-time existence conditions. However, we do have the following regularity result:
 \begin{theorem}
Let $f_t$ be a solution of the $H$-flow on $[0,T)$, for $0\leq T\leq \infty$. Suppose that $\lambda$ and $|\nabla^k\lambda|$ are uniformly bounded on $f_t(S)$ for any positive integer $k$.
If there exists a constant  $\alpha_0$ such that 
\begin{align*}
    \sup_{f_t(S)}|A|^2\leq \alpha_0 \quad \textrm{for $t\in[0,T)$ }.
\end{align*}
Then there exists $\alpha_k$ such that
\begin{align*}
    \sup_{f_t(S)}|\overline{\nabla}^kA|^2\leq \alpha_k \quad \textrm{for $t\in[0,T)$ }.
\end{align*} Also, if $T<\infty$, then
\begin{align*}
    \lim_{t\to T}\sup_{f_t(S)}|A|^2=\infty.
\end{align*}

 \end{theorem}

In general, we have to control the growth of $|\nabla^k\lambda|$ to get a long time existence. As a special case, Song and Weinkove \cite{SW} consider the differential form $\theta=\overline{\omega}_2+i\overline{\omega}_3$ and the space 
 \begin{align*}
     \mathcal{N}=\{f:S\to M| \textrm{ $f$ is an immersion and $f^\star(\theta)=\rho$}\}.
 \end{align*}
 They pointed out that 
 \begin{theorem}
 Let $f_t$ be a solution of the $H$-flow on $[0,T)$. If $f_0\in\mathcal{N}$, then $f_t\in\mathcal{N}$.
 \end{theorem}
 They call $H$-flow with initial data in $\mathcal{N}$ the special $H$-flow and it is similar to the Lagrangian mean curvature flow. In this special case, one can see that if the norm of the second fundamental form is uniformly bounded, $|\nabla^k\lambda|$ are uniformly bounded too.   

 Recall that the mean curvature flow is the gradient flow of the volume functional. It is thus natural to ask whether a local minimizer of the volume functional is stable under the mean curvature flow. In a series of paper from Tsai and Wang \cite{WT2,WT3,WT1}, they identified a strong stability condition on minimal submanifolds that implies uniqueness and dynamical stability properties. More precisely, they pointed out that the Jacobi operator of the second variation of the volume functional is $(\nabla^\perp)^*\nabla^\perp+\mathcal{R}-\mathcal{A}$, where $(\nabla^\perp)^*\nabla^\perp$ is the Bochner Laplacian of the normal bundle, $\mathcal{R}$ is an operator constructed from the restriction of the ambient Riemann curvature, and $\mathcal{A}$ is constructed from the second fundamental form. A minimal submanifold is said to be strongly stable if $\mathcal{R-A}$ is a positive operator. Suppose that the submanifold $\Gamma$ is $C^1$ close to the minimal, strongly stable surface $\Sigma$ then the mean curvature flow of $\Gamma$ will exist for all time and converge to $\Sigma$ smoothly. Moreover, they found that the zero section of the Atiyah-Hitchin manifold \cite{AH2} \cite{AH1} is strongly stable, allowing them to apply their stability result.

 In our case, we prove the dynamically stable results of the special $H$-flow. First, we show that strongly stable minimal surfaces with constant $\lambda$ are critical points of the hyperk\"{a}ler energy (\ref{E}) with negative second variation. Next, we observe that any complex Lagrangian surface in a hyperkähler 4-manifold is special Lagrangian and also minimal. This leads us to our main theorem:

 \begin{theorem}
Let $(M,\overline{g},I,J,K)$ be a hyperk\"{a}ler 4-manifold, $\Sigma\subset M$ be a compact, oriented, strongly stable complex Lagrangian surface with respect to $I,K$ ,i.e., $\omega_I+i\omega_K\equiv 0$ on $\Sigma$ and $\Gamma\subset M$ be a Lagranigan surface with respect to $K$ which is $C^1$ close to $\Sigma$. Consider 
\begin{align*}
    \mathcal{N}=\{f:\Gamma\to M| f^\star(\omega_J)=\rho, f^\star(\omega_K)=0 \textrm{ where $\rho$ is a given volume form on $\Gamma$}\},
\end{align*}
the special $H$-flow $f_t(\Gamma)=\Gamma^t\in \mathcal{N}$ with $f_0(\Gamma)=\Gamma$ exists for all time and converge to $\Sigma$ smoothly. 
\end{theorem}

As an example, we note that when $M$ is the total space of the cotangent bundle of a sphere, $T^* S^n$ (for $n > 1$), with the Stenzel metric \cite{St}, the zero section is strongly stable \cite{WT2}. Thus, we focus on the 4-dimensional case where the Stenzal metric is called the Eguchi-Hanson metric \cite{EH}. By computing the curvature carefully, we observe that the zero section of Eguchi-Hanson metric is not only strongly stable but also a complex Lagrangian surface. 

\begin{corollary}
Let $\Gamma$ be a Lagrangian surface with respect to $K$ in the Eguchi-Hanson space, which is $C^1$ close to the zero section $S^2$. Then the special H-flow $\Gamma^t$ with $\Gamma^0=\Gamma$ exists for all time and converges to $S^2$ smoothly.
\end{corollary}

This paper is organized as follows. In section 2, we discuss some basic properties of the $H-$flow. In section 3, we discuss strongly stable surface and the Eguchi-Hanson metric. In the last section, we prove the main result.

\textbf{Acknowledgements:} This work is written when the author is a math master student at National Taiwan university. I am grateful to my advisor Mao-Pei Tsui, professor Chung-Jun Tsai, professor Mu-Tao Wang. Their suggestions play an important role in this work.  

\section{Hyperkähler Flow}

\subsection{Introduction}

Let $S$ be a Riemann surface with volume form $\rho$ (symplectic form) and $(M, \ov{g}, I, J,K)$
is a hyperk\"{a}hler 4-manifold. 

\begin{defn}\label{D2.1}
Consider an immersion $f:S \rightarrow M$ and for $a=1,2,3$, we define 
\begin{itemize}
    \item $N_a = \frac{f^*(\ov{\omega}_a)}{\rho},$ where $\overline{w}_a$ is the K\"{a}hler forms with respect to $I,J,K$.
    \item $\xi_a$ is the Hamiltonian vector field  w.r.t $N_a$ ,i.e, $dN_a(\cdot)=\rho(\xi_a,\cdot)$.
\end{itemize}

The hyperk\"{a}hler mean curvature flow of Donaldson flow (in short $H$-flow) is given by
\begin{align}           
    \frac{\partial f}{\partial t}=If_\star(\xi_1)+Jf_\star(\xi_2)+Kf_\star(\xi_3).
\end{align}
\end{defn}

In the following, we use $i,j,k$ to denote 1 and 2, $\alpha$,$\beta$,$\gamma$ to denote 3 and 4, $A,B,C$ to denote 1,2,3,4. For any fixed point $p\in S$, let $\overline\nabla$ and $\nabla$ denote the Levi-Civita connections for $\overline{g}$ and $g=f^\star(\overline{g})$ respectively. We choose a normal coordinate system $\{x^1,x^2\}$ at $p$ with $\rho(\frac{\partial}{\partial x^1},\frac{\partial}{\partial x^2})=\rho(\partial_1,\partial_2)>0$
. Also, we select a normal coordinate system $\{y^1,y^2,y^3,y^4\}$ at $f(p)$ so that $\frac{\partial}{\partial y^i}=f_\star \partial_i$ i.e $\frac{\partial f^A}{\partial x^i}=\delta^A_i$ at $p$. Later, we also denote $\frac{\partial}{\partial y^i}$ by $\partial_i$. Then we can extend our normal coordinate to an local orthonormal basis $\{e_1,e_2\}$ for $Tf(S)$ and $\{e_3,e_4\}$ for $Nf(S)$. With respect to this orthonormal basis, we denote the second fundamatal form $A(\partial_i,\partial_j)=h_{\alpha ij} e_{\alpha}$ with 
\begin{align*}
    h_{\alpha ij}=\overline{g}(e_\alpha,\overline{\nabla}_{\partial_i}\partial_j),
\end{align*}
and the mean curvature vector at $p$ is given by $H=H_\alpha e_\alpha$ where $H_\alpha=g^{ij}h_{\alpha ij}$. Based on this coordinate system, Song and Weinkove rewrite the $H-$flow and their results are as follows. 

\begin{proposition}[\cite{SW}, Proposition 2.1]
Define a function $\lambda$ on S by $\lambda=\frac{d\mu}{\rho}$, where $d\mu$ is the induced volume form on $S$. 
The $H$-flow can be written as
\begin{align} \label{HF}
     \frac{\partial f}{\partial t}=\lambda \nabla\lambda+\lambda^2 H,
\end{align}
where $H$ is the mean curvature vector of $f(S)$ in $M$ and $\nabla$ is the connection of induced metric $g=f^\star(\overline{g})$ on S. Hence, the critical point of $f$ happens when $f(S)$ is a minimal surface and $\lambda$ is a constant. 
\end{proposition}

From (\ref{HF}), we see that $If_\star(\xi_1)+Jf_\star(\xi_2)+Kf_\star(\xi_3)$ can be decomposed to the normal part $\lambda^2 H$ and the tangential part $\lambda\nabla\lambda$ of $f(S)$ in $M$. Since this flow is similar to the mean curvature flow, we also derive a short-time existence. (See \cite{SW} and \cite{De} for more details.)

\begin{proposition}[\cite{SW}, Proposition 2.2]
Given any smooth initial map $f_0 \in \mathcal{M}$, there exists $T>0$ such that $ \frac{\partial f}{\partial t}=\lambda \nabla\lambda+\lambda^2 H$ admits a unique smooth solution $f_t \in \mathcal{M}$ for $t\in [0,T)$.
\end{proposition}

\subsection{Hyperkähler Energy}

\begin{defn}
Adopt the same notation in the Definition \ref{D2.1}, we define the hyperk\"{a}hler energy by
\begin{align*}
    E(f)= \sum_{a=1}^3 \| N_a \|^2_{L^2(S, \rho)}=\int_S \lambda^2\rho,
\end{align*}    
where $\lambda =\frac{d\mu}{\rho}$.
\end{defn}

In \cite{D}, Donaldson shows that $H$-flow is the gradient flow of the hyperk\"{a}hler energy. In the following, we are going to compute the first and second variation formulas of the hyperkähler energy to justify this result.

\begin{proposition}
Let $T=df(\frac{\partial}{\partial t})$ be the variational field which is compactly support then 
\begin{align}
    \frac{d}{dt}\Big|_{t=0}E=\int_{S}-2\overline{g}(\lambda\nabla\lambda,T)-2\lambda^2 \overline{g}(T,H) \rho.
\end{align}
Thus, the gradient flow of hyperk\"{a}hler energy is $H$-flow.
\end{proposition}
\begin{proof}
In local coordinate, the hyperk\"{a}hler energy is given by
\begin{align*}
    E=\int_{f_t(S)}\frac{\det\big(g(x,t)\big)}{\rho_{12}(x)}dx_1\wedge dx_2, \quad \textrm{since  } \lambda=\frac{\sqrt{\det(g(x,t))}}{\rho_{12}(x)}.
\end{align*}
Acting on its variational field $T=df(\frac{\partial}{\partial t})$ at $t=0$, we derive
\begin{align*}
    \overline{\nabla}_T\big(\frac{\det(g)}{\rho_{12} }\big)=\big[2\divg(T^t)-2\overline{g}(T,H)\big]\frac{\det(g)}{\rho_{12} }
\end{align*}
and
\begin{align*}
    \frac{\partial}{\partial t}\big|_{t=0}\log(\det g)=2\divg(T^t)-2\overline{g}(T,H).
\end{align*}
Therefore, 
\begin{align*}
      \frac{d}{dt}\Big|_{t=0}E&=\int_{f_0(S)}\lambda \big[2\divg(T^t)-2\overline{g}(T,H)\big]dA_0
      \\&=\int_{f_0(S)} 2\divg(\lambda T^t)-2\overline{g}(\nabla\lambda,T^t)-2\lambda \overline{g}(T,H)dA_0
      \\&=\int_{f_0(S)}-2\overline{g}(\nabla\lambda,T^t)-2\lambda \overline{g}(T,H)dA_0
      \\&=\int_{S}-2\overline{g}(\lambda\nabla\lambda,T)-2\lambda^2 \overline{g}(T,H) \rho.
\end{align*}
    
\end{proof}

\begin{proposition}
Let $T=df(\frac{\partial}{\partial t})$ is a variational field normal to $f_0(S)$ and compactly support
\begin{align*}
     &\frac{d^2}{dt^2}\Big|_{t=0}E
     \\&\kern1em\nonumber=\int_{S}\lambda^2\big[-2\sum_{i,j}\overline{g}(T,A_{ij})^2+2\overline{g}(\overline{R}(T,e_i)T,e_i)-2\overline{g}(\overline{\nabla}_TT,H)+2\sum_{i=1}^2\sum_{\alpha=3}^4\overline{g}(\overline{\nabla}_{e_i}T,e_\alpha)^2+4\overline{g}(T,H)^2\big]\rho
     \\&\kern3em\nonumber -\int_{S}2\lambda\overline{g}(\nabla\lambda,\overline{\nabla}_TT) \rho,
\end{align*}
where $\overline{R}$ denote the curvature of ambient manifold $M$.
In particular, if $f_0(S)$ is at the critical point of energy ,i.e., $H\equiv 0$ and $\lambda$ is a constant. Then,
\begin{align}
     \frac{d^2}{dt^2}\Big|_{t=0}E&=\int_{S}\lambda^2\Big[-2\sum_{i,j}\overline{g}(T,A_{ij})^2+2\overline{g}(\overline{R}(T,e_i)T,e_i)+2\sum_{i=1}^2\sum_{\alpha=3}^4\overline{g}(\overline{\nabla}_{e_i}T,e_\alpha)^2\Big]\rho.
\end{align}
\end{proposition}
\begin{proof}
 First, we consider two variational vectors $T=df(\frac{\partial}{\partial t})$ and $S=df(\frac{\partial}{\partial s})$
\begin{align*}
    \frac{\partial^2}{\partial s\partial t}\det\big(g(p,t,s)\big)&=\frac{\partial}{\partial s}\big(\det g g^{ij}\frac{\partial}{\partial t}g_{ij}\big)
    \\&=(\frac{\partial}{\partial s}g^{ij})(\frac{\partial}{\partial t}g_{ij})\det g+g^{ij}\det g (\frac{\partial^2}{\partial s\partial t}g_{ij})+\det g(g^{ij}\frac{\partial}{\partial s}g_{ij})(g^{kl}\frac{\partial}{\partial t}g_{kl}).
\end{align*}   
At $(p,0,0)$, $g_{ij}=\delta_{ij}$. So
\begin{align*}
    \frac{\partial^2}{\partial s\partial t}\Big|_{(0,0))}\det(g(p,t,s))=(\frac{\partial}{\partial s}g^{ij})(\frac{\partial}{\partial t}g_{ij})\det +\det g (\frac{\partial^2}{\partial s\partial t}g_{ii})+\det g(\frac{\partial}{\partial s}g_{ij})(\frac{\partial}{\partial t}g_{kl}).
\end{align*}   
We get 
\begin{align*}
    Sg^{ij}&=-g^{ik}Sg_{kl}g^{jl}=-g^{ik}(\overline{g}(\overline{\nabla}_{e_k}S,e_l)+\overline{g}(e_k,\overline{\nabla}_{e_l}S))g^{jl}
    \\&=-\overline{g}(\overline{\nabla}_{e_i}S,e_j)-\overline{g}(e_i,\overline{\nabla}_{e_j}S),
\end{align*}
\begin{align*}
    S\big(\overline{g}(\overline{\nabla}_{e_i}T,e_i)\big)&=\overline{g}(\overline{\nabla}_S\overline{\nabla}_{e_i}T,e_i)+\overline{g}(\overline{\nabla}_{e_i}T,\overline{\nabla}_{e_i}S)
    \\&=\overline{g}\big(R^N(S,e_i)T,e_i\big)+\overline{g}(\overline{\nabla}_{e_i}\overline{\nabla}_ST,e_i)+\overline{g}(\overline{\nabla}_{e_i}T,\overline{\nabla}_{e_i}S),
\end{align*}
and
\begin{align*}
    \frac{1}{4}Sg_{ii}Tg_{kk}=\overline{g}(\overline{\nabla}_{e_i}S,e_i)\overline{g}(\overline{\nabla}_{e_k}T,e_k).
\end{align*}
In conclusion, 
\begin{align*}
     \frac{\partial^2}{\partial s\partial t}\Big|_{(0,0)}\det\big(g(p,t,s)\big)&=-2(\overline{g}(\overline{\nabla}_{e_i}S,e_j)+\overline{g}(e_i,\overline{\nabla}_{e_j}S))\overline{g}(\overline{\nabla}_{e_i}T,e_j)+2\overline{g}(\overline{R}(S,e_i)T,e_i)
    \\&\kern1em+2\overline{g}(\overline{\nabla}_{e_i}\overline{\nabla}_ST,e_i)+2\overline{g}(\overline{\nabla}_{e_i}T,\overline{\nabla}_{e_i}S)+4\overline{g}(\overline{\nabla}_{e_i}S,e_i)\overline{g}(\overline{\nabla}_{e_k}T,e_k).
\end{align*}
Take $S=T$ and $T$ is normal to $f_0(S)$,
\begin{align*}
    \frac{\partial^2}{\partial t^2}\Big|_{t=0}\det(g(p,t))&=-2(\overline{g}(\overline{\nabla}_{e_i}T,e_j)+\overline{g}(e_i,\overline{\nabla}_{e_j}T))\overline{g}(\overline{\nabla}_{e_i}T,e_j)+2\overline{g}(\overline{R}(T,e_i)T,e_i)
    \\&\kern1em+2\overline{g}(\overline{\nabla}_{e_i}\overline{\nabla}_TT,e_i)+2\overline{g}(\overline{\nabla}_{e_i}T,\overline{\nabla}_{e_i}T)+4\overline{g}(\overline{\nabla}_{e_i}T,e_i)\overline{g}(\overline{\nabla}_{e_k}T,e_k)
    \\&=-4\sum_{i,j}\overline{g}(T,A_{ij})^2+2\overline{g}(\overline{R}(T,e_i)T,e_i)+2\divg(\overline{\nabla}_TT)^t-2\overline{g}(\overline{\nabla}_TT,H)
    \\&\kern1em+2\sum_{i=1}^2|\overline{\nabla}_{e_i}T|^2+4\overline{g}(T,H)^2.
\end{align*}
Note that $e_{3},e_4$ are normal vectors, we can compute that
\begin{align*}
    \sum_{i=1}^2|\overline{\nabla}_{e_i}T|^2=\sum_{i=1}^2[\sum_{j=1}^2\overline{g}(\overline{\nabla}_{e_i}T,e_j)^2+\sum_{\alpha=3}^4\overline{g}(\overline{\nabla}_{e_i}T,e_\alpha)^2]=\overline{g}(T,A_{ij})^2+  \sum_{i=1}^2\sum_{\alpha=3}^4\overline{g}(\overline{\nabla}_{e_i}T,e_\alpha)^2.
\end{align*}
Our results are followed by Stoke's theorem.
\end{proof}

\subsection{Evolution Formulas}

To discuss some properties of $H$-flow, we first derive some evolution formulas. Recall that $S$ is a Riemann surface with volume form $\rho$ and $(M, \ov{g}, I, J,K)$ is a hyperk\"{a}hler 4-manifold. The hyperk\"{a}hler flow is an immersion $f:S\to M$ evolved by 
\begin{align*}
    \frac{\partial}{\partial t}f=\lambda\nabla\lambda+\lambda^2 H,
\end{align*}
where $H$ is the mean curvature vector of $f(S)$ in $M$ and $\nabla$ is the connection of the induced metric $g=f^\star(\overline{g})$ on $S$. 
\begin{proposition}[\cite{SW}, Proposition 2.3] \label{P2.7} Along the $H$-flow,
\begin{align*}  
\frac{\partial}{\partial t}g_{kl}&=\nabla_k\nabla_l(\lambda^2)-2\lambda^2H_\alpha h_{\alpha kl},
\\ \ddt{} d\mu &= \left( \Delta ( \frac{\lambda^2}{2}
) - \lambda^2 |H|^2 \right) d\mu,
\\\ddt{} (\lambda^2) &= \lambda^2\left( \Delta ( \lambda^2
) - 2\lambda^2 |H|^2 \right).
\end{align*}
\end{proposition}

Due to the \Cref{P2.7}, one can see that $\lambda$ is non-increasing along the flow. Next, we derive the evolution formula of the second fundamental form.

\begin{theorem}
The second fundamental form is evolved by
\begin{align}\label{EVh}
    \frac{\partial}{\partial t}h_{\alpha ij}\nonumber&=\lambda^2\{\Delta h_{\alpha ij}-H_{\beta}h_{\beta il}h_{\alpha jl}+(\overline{\nabla}_k\overline{R})_{jki\alpha}+(\overline{\nabla}_j\overline{R})_{ikk\alpha}
    \\&\kern1em\nonumber-h_{\alpha il}\overline{R}_{jkkl}-h_{\alpha jl}\overline{R}_{ikkl}+h_{\beta ji}\overline{R}_{\beta kk\alpha}+2h_{\beta jk}\overline{R}_{ik\beta \alpha}+2h_{\beta ik}\overline{R}_{jk\beta\alpha}-2h_{\alpha kl}\overline{R}_{jkil}
    \\&\kern1em\nonumber+h_{\beta ik}(h_{\beta lk}h_{\alpha lj}-h_{\beta lj}h_{\alpha lk})+h_{\alpha mk}(h_{\gamma ji}h_{\gamma km}-h_{\gamma jm}h_{\gamma ik})+h_{\alpha im}(h_{\gamma jk}h_{\gamma km}-h_{\gamma jm}h_{\gamma kk})\}
    \\&\kern1em\nonumber+\lambda\lambda_k\overline{g}(\overline{\nabla}_{e_k}e_\alpha,e_\beta)h_{\beta ij}+\lambda^2  \overline{g}(\overline{\nabla}_{H}e_\alpha,\overline{\nabla}_{e_j}e_i)+(\lambda_j\lambda_k+\lambda\lambda_{jk})h_{\alpha ik}+(\lambda_i\lambda_k+\lambda\lambda_{ik})h_{\alpha jk}+\lambda\lambda_k h_{\alpha ij,k}
    \\&\kern1em+2(\lambda_i\lambda_j+\lambda\lambda_{ij})H_\alpha+2\lambda\lambda_iH_{\alpha,j}+2\lambda\lambda_jH_{\alpha,i},
\end{align}
where $\lambda_k=e_k(\lambda)$, $\lambda_{ik}=e_ie_k(\lambda)$ and $\overline{R}$ denote the curvature of $M$. 
\end{theorem}
\begin{proof}
We compute
\begin{align*}
    \frac{\partial}{\partial t}h_{\alpha ij}&=\overline{g}(\overline{\nabla}_{\nabla(\frac{\lambda^2}{2})+\lambda^2 H}e_\alpha,\overline{\nabla}_{e_j}e_i)+\overline{g}(e_\alpha,\overline{\nabla}_{\nabla(\frac{\lambda^2}{2})+\lambda^2 H}\overline{\nabla}_{e_j}e_i).
\end{align*}  
The first term is 
\begin{align*}
    \overline{g}(\overline{\nabla}_{\nabla(\frac{\lambda^2}{2})+\lambda^2 H}e_\alpha,\overline{\nabla}_{e_j}e_i)&=\overline{g}(\overline{\nabla}_{\nabla(\frac{\lambda^2}{2})}e_\alpha,\overline{\nabla}_{e_j}e_i)+\overline{g}(\overline{\nabla}_{\lambda^2 H}e_\alpha,\overline{\nabla}_{e_j}e_i)
    \\&=\lambda\lambda_k\overline{g}(\overline{\nabla}_{e_k}e_\alpha,e_\beta)h_{\beta ij}+\lambda^2 \overline{g}(\overline{\nabla}_{H}e_\alpha,\overline{\nabla}_{e_j}e_i),
\end{align*}
and the second term is 
\begin{align*}
\overline{g}(e_\alpha,\overline{\nabla}_{\nabla(\frac{\lambda^2}{2})+\lambda^2 H}\overline{\nabla}_{e_j}e_i)&=\overline{g}(e_\alpha,\overline{\nabla}_{e_j}\overline{\nabla}_{\nabla(\frac{\lambda^2}{2})+\lambda^2 H}e_i)+\overline{g}\big(\overline{R}(\nabla(\frac{\lambda^2}{2})+\lambda^2 H,e_j)e_i,e_\alpha\big)
\\&=\overline{g}\big(e_\alpha,\overline{\nabla}_{e_j}\overline{\nabla}_{e_i}[\nabla(\frac{\lambda^2}{2})+\lambda^2 H]\big)+\lambda\lambda_k\overline{R}_{kji \alpha}+\lambda^2 H_{\beta}\overline{R}_{\beta ji  \alpha}.
\end{align*}
Here 
\begin{align*}
    \overline{g}(e_\alpha,\overline{\nabla}_{e_j}\overline{\nabla}_{e_i}\nabla(\frac{\lambda^2}{2}))&=(\lambda_j\lambda_k+\lambda\lambda_{jk})h_{\alpha ik}+(\lambda_i\lambda_k+\lambda\lambda_{ik})h_{\alpha jk}+\lambda\lambda_k\overline{g}(e_\alpha,\overline{\nabla}_{e_j}\overline{\nabla}_{e_i}e_k)
    \\&=(\lambda_j\lambda_k+\lambda\lambda_{jk})h_{\alpha ik}+(\lambda_i\lambda_k+\lambda\lambda_{ik})h_{\alpha jk}+\lambda\lambda_k h_{\alpha ik,j},
\end{align*}
and
\begin{align*}
    \overline{g}(e_\alpha,\overline{\nabla}_{e_j}\overline{\nabla}_{e_i} (\lambda^2 H))&=\overline{g}(e_\alpha,e_je_i(\lambda^2 )H+e_i(\lambda^2 )\overline{\nabla}_{e_j}H+e_j(\lambda^2 )\overline{\nabla}_{e_i}H+\lambda^2 \overline{\nabla}_{e_j}\overline{\nabla}_{e_i}H)
    \\&=2(\lambda_i\lambda_j+\lambda\lambda_{ij})H_\alpha+2\lambda\lambda_iH_{\alpha,j}+2\lambda\lambda_jH_{\alpha,i}+\lambda^2\overline{g}(e_\alpha,\overline{\nabla}_{e_j}\overline{\nabla}_{e_i}H).
\end{align*}
Since
\begin{align*}
\overline{g}(e_\alpha,\overline{\nabla}_{e_j}\overline{\nabla}_{e_i}H)&=\overline{g}(e_\alpha,\overline{\nabla}_{e_j}\overline{\nabla}^T_{e_i}H)+\overline{g}(e_\alpha,\overline{\nabla}_{e_j}\overline{\nabla}^N_{e_i}H)
\\&=-\overline{g}(\overline{\nabla}_{e_j}e_\alpha,\overline{\nabla}^T_{e_i}H)+H_{\alpha,ij}
\\&=h_{\alpha jl}\overline{g}(\overline{\nabla}_{e_i}H,e_l)+H_{\alpha,ij}
\\&=-h_{\alpha jl}\overline{g}(H,\overline{\nabla}_{e_i}e_l)+H_{\alpha,ij}
\\&=-H_{\beta}h_{\beta il}h_{\alpha jl}+H_{\alpha,ij},
\end{align*}
we conclude that
\begin{align*}
     \frac{\partial}{\partial t}h_{\alpha ij}&=\lambda\lambda_k\overline{g}(\overline{\nabla}_{e_k}e_\alpha,e_\beta)h_{\beta ij}+\lambda^2 \overline{g}(\overline{\nabla}_{H}e_\alpha,\overline{\nabla}_{e_j}e_i)+\lambda\lambda_k\overline{R}_{kji \alpha}+\lambda^2 H_{\beta}\overline{R}_{\beta ji  \alpha}
     \\&\kern1em+(\lambda_j\lambda_k+\lambda\lambda_{jk})h_{\alpha ik}+(\lambda_i\lambda_k+\lambda\lambda_{ik})h_{\alpha jk}+\lambda\lambda_k h_{\alpha ik,j}
     \\&\kern1em+2(\lambda_i\lambda_j+\lambda\lambda_{ij})H_\alpha+2\lambda\lambda_iH_{\alpha,j}+2\lambda\lambda_jH_{\alpha,i}+\lambda^2(H_{\alpha,ij}-H_{\beta}h_{\beta il}h_{\alpha jl}).
\end{align*}
Using the Laplacian of the second fundamental form \cite{H1,H2}, \begin{align*}
     \Delta h_{\alpha ij}\nonumber&=H_{\alpha,ij}-(\overline{\nabla}_k\overline{R})_{jki\alpha}-(\overline{\nabla}_j\overline{R})_{ikk\alpha}
     \\&\kern1em\nonumber+h_{\alpha il}\overline{R}_{jkkl}+h_{\alpha jl}\overline{R}_{ikkl}-h_{\beta kk}\overline{R}_{j\beta i\alpha}-h_{\beta ji}\overline{R}_{\beta kk\alpha}-2h_{\beta jk}\overline{R}_{ik\beta \alpha}-2h_{\beta ik}\overline{R}_{jk\beta\alpha}+2h_{\alpha kl}\overline{R}_{jkil}
     \\&\kern1em-h_{\beta ik}(h_{\beta lk}h_{\alpha lj}-h_{\beta lj}h_{\alpha lk})-h_{\alpha mk}(h_{\gamma ji}h_{\gamma km}-h_{\gamma jm}h_{\gamma ik})
    -h_{\alpha im}(h_{\gamma jk}h_{\gamma km}-h_{\gamma jm}h_{\gamma kk})
\end{align*}
we replace $H_{\alpha ij}$ by $\Delta h_{\alpha ij}$. Then, the proof is completed by using the Codazzi equation $h_{\alpha ik,j}+\overline{R}_{kji\alpha}=h_{\alpha ij,k}$.
\end{proof}

\begin{corollary}
The norm of the second fundamental form $|A|^2=g^{ik}g^{jl}h_{\alpha ij}h_{\alpha kl}$ is evolved by
\begin{align*}
   \frac{\partial}{\partial t}|A|^2&\nonumber=\lambda^2\Big\{\Delta|A|^2-2|\nabla A|^2+2h_{\alpha ij}[(\overline{\nabla}_k\overline{R})_{jki\alpha}+(\overline{\nabla}_j\overline{R})_{ikk\alpha}]
 \\&\kern1em\nonumber+h_{\alpha ij}\big[-4h_{\alpha il}\overline{R}_{jkkl}+8h_{\beta jk}\overline{R}_{ik\beta \alpha}+2h_{\beta ji}\overline{R}_{\beta kk\alpha}-4h_{\alpha kl}\overline{R}_{jkil}\big]
 \\&\kern1em\nonumber+2\sum_{\alpha,\beta,i,l}(\sum_{k}h_{\alpha ik}h_{\beta kl}-h_{\alpha kl}h_{\beta ik})+2\sum_{i,j,k,l}(\sum_\alpha h_{\alpha ij}h_{\alpha kl})^2\Big\}
 \\&\kern1em+2h_{\alpha ij}\big[\lambda\lambda_k h_{\alpha ij,k}+2(\lambda_i\lambda_j+\lambda\lambda_{ij})H_\alpha+4\lambda\lambda_iH_{\alpha,j}\big].
\end{align*}
\end{corollary}  
\begin{proof}
From the definition, 
\begin{align*}
    \frac{\partial}{\partial t}|A|^2=2(\frac{\partial}{\partial t}g^{ik})h_{\alpha ij}h_{\alpha kj}+2(\frac{\partial}{\partial t}h_{\alpha ij})h_{\alpha ij}.
\end{align*}
Note that
\begin{align*}
    \frac{\partial}{\partial t}g^{ik}&=-g^{ij}(\frac{\partial}{\partial t}g_{jl})g^{kl}=-g^{ij}(\nabla_j\nabla_l(\lambda^2)-2\lambda^2H_\alpha h_{\alpha jl})g^{kl}
    \\&=-2(\lambda_i\lambda_k+\lambda\lambda_{ik})+2\lambda^2H_\alpha h_{\alpha ik}.
\end{align*}    
The first part is 
\begin{align*}
    2(\frac{\partial}{\partial t}g^{ik})h_{\alpha ij}h_{\alpha kj}=-4(\lambda_i\lambda_k+\lambda\lambda_{ik})h_{\alpha ij}h_{\alpha kj}+4\lambda^2H_\alpha h_{\alpha ik}h_{\alpha ij}h_{\alpha kj},
\end{align*}
and the second part is 
\begin{align*}
    &2(\frac{\partial}{\partial t}h_{\alpha ij})h_{\alpha ij}
    \\&\kern1em=2\lambda^2\Big\{ h_{\alpha ij}\Delta h_{\alpha ij}-H_{\beta}h_{\beta il}h_{\alpha jl}h_{\alpha ij}+h_{\alpha ij}\big[(\overline{\nabla}_k\overline{R})_{jki\alpha}+(\overline{\nabla}_j\overline{R})_{ikk\alpha}\big] 
    \\&\kern1em-h_{\alpha ij}\big[2h_{\alpha il}\overline{R}_{jkkl}+h_{\beta ji}\overline{R}_{\beta kk\alpha}+4h_{\beta jk}\overline{R}_{ik\beta \alpha}-2h_{\alpha kl}\overline{R}_{jkil}\big]
    \\&\kern1em+h_{\alpha ij}\big[h_{\beta ik}(h_{\beta lk}h_{\alpha lj}-h_{\beta lj}h_{\alpha lk})+h_{\alpha mk}(h_{\gamma ji}h_{\gamma km}-h_{\gamma jm}h_{\gamma ik})+h_{\alpha im}(h_{\gamma jk}h_{\gamma km}-h_{\gamma jm}h_{\gamma kk})\big]\Big\}
    \\&\kern1em+4h_{\alpha ij}(\lambda_j\lambda_k+\lambda\lambda_{jk})h_{\alpha ik}+2h_{\alpha ij}\big[\lambda\lambda_kh_{\alpha ij,k}+2(\lambda_i\lambda_j+\lambda\lambda_{ij})H_\alpha+4\lambda\lambda_iH_{\alpha,j}\big],
\end{align*}
where we used (\ref{EVh}). Since
\begin{align*}
    \Delta |A|^2=2|\nabla A|^2+2h_{\alpha ij}\Delta h_{\alpha ij},
\end{align*}
we derive that  
\begin{align*}
 \frac{\partial}{\partial t}|A|^2&=\lambda^2\Big\{\Delta|A|^2-2|\nabla A|^2+2h_{\alpha ij}\big[(\overline{\nabla}_k\overline{R})_{jki\alpha}+(\overline{\nabla}_j\overline{R})_{ikk\alpha}\big]
 \\&\kern1em+h_{\alpha ij}\big[-4h_{\alpha il}\overline{R}_{jkkl}+8h_{\beta jk}\overline{R}_{ik\beta \alpha}+2h_{\beta ji}\overline{R}_{\beta kk\alpha}-4h_{\alpha kl}\overline{R}_{jkil}\big]
 \\&\kern1em+2h_{\alpha ij}\big[h_{\beta ik}(h_{\beta lk}h_{\alpha lj}-h_{\beta lj}h_{\alpha lk})+h_{\alpha mk}(h_{\gamma ji}h_{\gamma km}-h_{\gamma jm}h_{\gamma ik})
 \\&\kern1em+h_{\alpha im}h_{\gamma jk}h_{\gamma km}\big]\Big\}+2h_{\alpha ij}\big[\lambda\lambda_k h_{\alpha ij,k}+2(\lambda_i\lambda_j+\lambda\lambda_{ij})H_\alpha+4\lambda\lambda_iH_{\alpha,j}\big].
\end{align*} 
Note that
\begin{align*}
    \nonumber h&_{\alpha ij}\big[h_{\beta ik}(h_{\beta lk}h_{\alpha lj}-h_{\beta lj}h_{\alpha lk})+h_{\alpha mk}(h_{\gamma ji}h_{\gamma km}-h_{\gamma jm}h_{\gamma ik})
+h_{\alpha im}h_{\gamma jk}h_{\gamma km}\big]
\\&\nonumber=2h_{\alpha ij}h_{\alpha il}h_{\beta jk}h_{\beta kl}-2h_{\alpha ij}h_{\alpha kl}h_{\beta jl}h_{\beta ik}+h_{\alpha ij}h_{\alpha kl}h_{\beta ij}h_{\beta kl}
\\&=\sum_{\alpha,\beta,i,l}(\sum_{k}h_{\alpha ik}h_{\beta kl}-h_{\alpha kl}h_{\beta ik})+\sum_{i,j,k,l}(\sum_\alpha h_{\alpha ij}h_{\alpha kl})^2,
\end{align*}
we then complete the proof.
\end{proof}

\subsection{Special Hyperkähler Flow}

Consider the $I-$holomorphic (2,0)-form $\theta=\overline{\omega}_2+i\overline{\omega}_3$ and the space

$$\mathcal{N} = \{ f : S \rightarrow M \  | \ f \textrm{ is an immersion
and} \ f^*(\theta) = \rho \}.$$
This condition indicates that the immersion surface is Lagrangian with respect to $\overline{\omega}_3$ and is symplectic with respect to $\overline{\omega}_2$. In \cite{SW}, Song and Weinkove proved the following.

\begin{theorem}[\cite{SW}, Theoerm 1]\label{T2.10}
Let $f_t$ be a solution of the $H$-flow for $0\le
t \le T$.   If $f_0\in \mathcal{N}$, then $f_t \in \mathcal{N}$ for $0\le t \le T$. Thus, we call the $H-$flow with initial point in $\mathcal{N}$ by the special hyperkähler flow (special $H-$flow).
\end{theorem}

 Define the functions $\eta_i$ on $S$ by 
\begin{align*}
    \eta_i=\frac{f^\star\overline{\omega}_i}{d\mu}, 
\end{align*}
where $i=1,2,3$. In other words, $f\in \mathcal{N}\Longleftrightarrow \eta_2=\frac{1}{\lambda}, \eta_3=0$. To briefly explain the proof of \Cref{T2.10}, let me compute the evolution formula.

\begin{proposition}\label{P2.11}
 Let $\overline{\omega}$ be a 2-form and $\eta=\frac{f^*\overline{\omega}}{d\mu}$. Then, $\eta$ evolves by 
\begin{align}
      \frac{\partial}{\partial t} \eta&\nonumber=\lambda^2[\Delta\eta+\overline{\omega}((\overline{R}(e_1,e_k)e_k)^N,e_2)+\overline{\omega}(e_1,(\overline{R}(e_2,e_k)e_k)^N)+\eta|A|^2]
    \\&\kern1em\nonumber+2\lambda[e_1(\lambda)\overline{\omega}(H,e_2)+e_2(\lambda)\overline{\omega}(e_1,H)]
    \\&\kern1em\nonumber+\lambda e_i(\lambda)[\overline{\omega}(\overline{\nabla}_{e_1}e_i,e_2)+\overline{\omega}(e_1,\overline{\nabla}_{e_2}e_i)]-2\lambda^2\overline{\omega}_{34}(h_{3k1}h_{4k2}-h_{3k2}h_{4k1})
    \\&\kern1em+(\overline{\nabla}_{\lambda\nabla\lambda+\lambda^2H}\overline{\omega})(e_1,e_2)-\lambda^2\overline{\omega}_{12,kk}-2\lambda^2\overline{\omega}_{\alpha 2,k}h_{\alpha k1}-2\lambda^2\overline{\omega}_{1\alpha ,k}h_{\alpha k2},  
\end{align}
where $\{e_1,e_2\}$ are orthonormal frame for $Tf(S)$ and $\{e_3,e_4\}$ are orthonormal frame for $(Tf(S))^\perp$. 
\end{proposition}

\begin{proof}
\begin{align*}
    \frac{\partial}{\partial t} f^\star\overline{\omega}(e_1,e_2)&=\big(\overline{\nabla}_{\lambda\nabla\lambda+\lambda^2H}\overline{\omega}\big)(e_1,e_2)+\overline{\omega}(\overline{\nabla}_{\lambda\nabla\lambda+\lambda^2 H}e_1,e_2)+\overline{\omega}(e_1,\overline{\nabla}_{\lambda\nabla\lambda+\lambda^2 H}e_2)
    \\&=\big(\overline{\nabla}_{\lambda\nabla\lambda+\lambda^2H}\overline{\omega}\big)(e_1,e_2)+\overline{\omega}\big(\overline{\nabla}_{e_1}{(\lambda\nabla\lambda+\lambda^2 H)},e_2\big)+\overline{\omega}\big(e_1,\overline{\nabla}_{e_2}{(\lambda\nabla\lambda+\lambda^2 H})\big).
\end{align*}
First, we have
\begin{align*}
    \overline{\omega}\big(\overline{\nabla}_{e_1}(\lambda\nabla\lambda),e_2\big)&=\overline{\omega}\big(\overline{\nabla}_{e_1}\nabla(\frac{\lambda^2}{2}),e_2\big)=\overline{\omega}\big(\overline{\nabla}_{e_1}\nabla_{e_i}(\frac{\lambda^2}{2})e_i,e_2\big)+\overline{\omega}\big(\nabla_{e_i}(\frac{\lambda^2}{2})\overline{\nabla}_{e_1}e_i,e_2\big)
    \\&=\nabla_{e_1}\nabla_{e_1}(\frac{\lambda^2}{2})\eta+\lambda e_i(\lambda)\overline{\omega}(\overline{\nabla}_{e_1}e_i,e_2),
\end{align*}
and 
\begin{align*}
    \overline{\omega}\big(\overline{\nabla}_{e_1}(\lambda^2 H),e_2\big)=2\lambda\overline{\omega}\big(e_1(\lambda)H,e_2)+\lambda^2\overline{\omega}(\overline{\nabla}_{e_1}H,e_2\big)=2\lambda\overline{\omega}(e_1(\lambda)H,e_2)+ \lambda^2\overline{\omega}\big((\overline{\nabla}_{e_1}H)^N,e_2\big)-\lambda^2\eta \overline{g}\big(H,\overline{\nabla}_{e_1}e_1\big).
\end{align*}
So, 
\begin{align*}
    \overline{\omega}&(\overline{\nabla}_{e_1}(\lambda^2 H),e_2)+\overline{\omega} (e_1,\overline{\nabla}_{e_2}(\lambda^2 H))
    \\&=2\lambda\big[e_1(\lambda)\overline{\omega}(H,e_2)+e_2(\lambda)\overline{\omega}(e_1,H)\big]+\lambda^2\big[\overline{\omega}((\overline{\nabla}_{e_1}H)^N,e_2)+\overline{\omega}(e_1,(\overline{\nabla}_{e_2}H)^N)\big]-\lambda^2\eta|H|^2.
\end{align*}
Recall that the Laplacian of 2-form (\cite{W3}) is given by
\begin{align}
    \Delta \eta&\nonumber=\overline{\omega}_{12,kk}+2\overline{\omega}_{\alpha 2,k}h_{\alpha k1}+2\overline{\omega}_{1\alpha ,k}h_{\alpha k2}+2\overline{\omega}_{34}(h_{3k1}h_{4k2}-h_{3k2}h_{4k1})-\eta|A|^2
    \\&\kern1em+\overline{R}_{k1k\alpha}\overline{\omega}_{\alpha 2}+\overline{R}_{k2k\alpha}\overline{\omega}_{1\alpha }+H_{\alpha,1}\overline{\omega}_{\alpha 2}+H_{\alpha,2}\overline{\omega}_{1\alpha }.
\end{align}
We deduce that

\begin{align*}
   \overline{\omega}&((\overline{\nabla}_{e_1}H)^N,e_2)+\overline{\omega}(e_1,(\overline{\nabla}_{e_2}H)^N)&
   \\&=\Delta\eta-\overline{R}_{k1k\alpha}\overline{\omega}_{\alpha 2}-\overline{R}_{k2k\alpha}\overline{\omega}_{1\alpha }
  +\eta|A|^2-2\overline{\omega}_{34}(h_{3k1}h_{4k2}-h_{3k2}h_{4k1})-\overline{\omega}_{12,kk}-2\overline{\omega}_{\alpha 2,k}h_{\alpha k1}-2\overline{\omega}_{1\alpha ,k}h_{\alpha k2}.
\end{align*}
Thus,
\begin{align*}  
    \frac{\partial}{\partial t} f^\star\overline{\omega}(e_1,e_2)&=\lambda^2\big[\Delta\eta-\overline{R}_{k1k\alpha}\overline{\omega}_{\alpha 2}-\overline{R}_{k2k\alpha}\overline{\omega}_{1\alpha }-\eta|H|^2+\eta|A|^2\big]
    \\&\kern1em+2\lambda\big[e_1(\lambda)\overline{\omega}(H,e_2)+e_2(\lambda)\overline{\omega}(e_1,H)\big]+\Delta(\frac{\lambda^2}{2})\eta
    \\&\kern1em+\lambda e_i(\lambda)\big[\overline{\omega}(\overline{\nabla}_{e_1}e_i,e_2)+\overline{\omega}(e_1,\overline{\nabla}_{e_2}e_i)\big]-2\lambda^2\overline{\omega}_{34}(h_{3k1}h_{4k2}-h_{3k2}h_{4k1})
    \\&\kern1em+(\overline{\nabla}_{\lambda\nabla\lambda+\lambda^2H}\overline{\omega})(e_1,e_2)-\lambda^2\overline{\omega}_{12,kk}-2\lambda^2\overline{\omega}_{\alpha 2,k}h_{\alpha k1}-2\lambda^2\overline{\omega}_{1\alpha ,k}h_{\alpha k2}.
\end{align*}
Finally, we know that 
\begin{align*}
    \frac{\partial}{\partial t}\eta=\frac{\frac{\partial}{\partial t} f^\star\overline{\omega}(e_1,e_2)}{d\mu}-\frac{ f^\star\overline{\omega}}{(d\mu)^2}\frac{\partial}{\partial t}d\mu,
\end{align*}
and 
\begin{align*}
    \ddt{} d\mu = \left( \Delta ( \frac{\lambda^2}{2}
) - \lambda^2 |H|^2 \right) d\mu.
\end{align*}
We complete the proof.    
\end{proof}

By \Cref{P2.11}, we compute $\frac{\partial}{\partial t}\eta_3$ and $\frac{\partial}{\partial t}(\eta_2-\frac{1}{\lambda})$. Using the maximum principle, we conclude that these two terms vanish all time. Thus, the proof of \Cref{T2.10} follows. (See \cite{SW} for more details.)

For any $p\in S$, let $\{x_1,x_2\}$ be a normal coordinate system at $p$ as before, and we take $\nu_i=K\partial_i$, $i=1,2$. Due to \Cref{T2.10}, 
\begin{align*}
    \overline{g}(\nu_i, \partial_j)=\overline{g}(K\partial_i,\partial_j)=\eta_3(\partial_i,\partial_j)=0.
\end{align*}
Then, we have an orthonormal basis $\{\partial_1,\partial_2,\nu_1,\nu_2\}$ for $T_{f(p)}M$ and the second fundamental form is defined by $h_{ijk}=\overline{g}(\nu_i,\overline{\nabla}_{\partial_i}\partial_j)$. Moreover, we deduce the following. 
\begin{lemma} 
At $p$,
\begin{align}
   &\nonumber \overline{\omega}_1=\begin{pmatrix} 0 & \eta_1 & 0 & -\sqrt{1-\eta_1^2}\\
- \eta_1& 0 & \sqrt{1-\eta_1^2} & 0 \\
0 & -\sqrt{1-\eta_1^2} & 0 & - \eta_1 \\
\sqrt{1-\eta_1^2} & 0 &  \eta_1 & 0
\end{pmatrix},
\\&\nonumber
\overline{\omega}_2=\begin{pmatrix} 0 & \sqrt{1-\eta_1^2} & 0 & \eta_1\\
- \sqrt{1-\eta_1^2}& 0 & -\eta_1 & 0 \\
0 & \eta_1 & 0 & -\sqrt{1-\eta_1^2} \\
-\eta_1 & 0 &  \sqrt{1-\eta_1^2} & 0
\end{pmatrix},
\\&\overline{\omega}_3=\begin{pmatrix} 0 & 0 & 1 & 0 \\
0& 0 & 0 & 1\\   
- 1 & 0 & 0 & 0 \\
0 & -1 & 0 & 0
\end{pmatrix}.
\end{align}
\end{lemma}

\begin{proof}
We compute
\begin{align*}
    &\overline{\omega}_1(\partial_i,\nu_j)=\overline{g}(I\partial_i,K\partial_j)=-\overline{g}(J\partial_i,\partial_j)=-\overline{\omega}_2(\partial_i,\partial_j), \\&\overline{\omega}_1(\nu_i,\nu_j)=\overline{g}(IK\partial_i,K\partial_j)=\overline{g}(I\partial_j,\partial_i)=\overline{\omega}_1(\partial_j,\partial_i).
\end{align*}
Note that $\eta_1^2=1-\frac{1}{\lambda^2}\Rightarrow \frac{1}{\lambda}=\pm \sqrt{1-\eta_1^2}$ and $    \rho(\partial_1,\partial_2)=\eta_2(\partial_1,\partial_2)=\frac{1}{\lambda}>0$, we derive that $\frac{1}{\lambda}=\sqrt{1-\eta_1^2}$.

\end{proof}

In term of this coordinate system, we observe that

\begin{align}\label{Lam}
    \nonumber\partial_k(\lambda)&=-\lambda^2[\overline{\omega}_2(\overline{\nabla}_{\partial_k}\partial_1,\partial_2)+\overline{\omega}_2(\partial_1,\overline{\nabla}_{\partial_k}\partial_2)]=-\lambda^2\eta_1(h_{1k1}+h_{2k2})=-\lambda^2\eta_1H_k,
    \\\nonumber \partial_j\eta_1&=-\frac{1}{\lambda}H_j, 
    \\\partial_j\partial_k(\lambda)&=(2\lambda^3-\lambda)H_jH_k-\lambda^2\eta_1H_{k,j}.
\end{align}
We then have a simplified evolution formula of the second fundamental form.
\begin{corollary} Using the above coordinate, the evolution formula second fundamental form of special $\textrm{hyperk}\Ddot{a}\textrm{hler}$ flow is given by
\begin{align*}
    \frac{\partial}{\partial t}h_{ijk}&\nonumber=\lambda^2\Big\{ \Delta h_{ijk}+(\overline{\nabla}_l\overline{R})_{jlk\overline{i}}+(\overline{\nabla}_j\overline{R})_{kll\overline{i}}+h_{ikl}\overline{R}_{jmml}+h_{ijl}\overline{R}_{kmml}+h_{ljk}\overline{R}_{lmmi}
    \\&\nonumber\kern1em+2h_{l jm}\overline{R}_{kmli}+2h_{lkm}\overline{R}_{jmli}+2h_{iml}\overline{R}_{jmlk}+h_{mkr}(h_{mlr}h_{ilj}-h_{mlj}h_{ilr})+h_{iml}(h_{rjk}h_{rlm}-h_{rjm}h_{rkl})
    \\&\nonumber\kern1em +h_{ikm}(h_{rjl}h_{rlm}-h_{rjm}H_{r})-(H_mh_{mil}h_{jkl}+H_mh_{mkl}h_{ijl})
    \\&\nonumber\kern1em-\lambda\eta_1(H_lh_{ijk,l}+h_{jkl}H_{l,i}++h_{ijl}H_{l,k}+h_{ikl}H_{l,j}+2H_kH_{i,j}+2H_jH_{i,k}+2H_iH_{k,j})
    \\&\kern1em+(3\lambda^2-2)(H_iH_lh_{jkl}+H_kH_lh_{ijl}+H_jH_lh_{ikl}+2H_iH_jH_k)\Big\}.
\end{align*}
\end{corollary}

\begin{corollary}
The norm of the second fundamental form evolves by  
\begin{align}\label{A}
  \frac{\partial}{\partial t}|A|^2&\nonumber=\lambda^2\Big\{\Delta|A|^2-2|\nabla A|^2+2h_{ijk}((\overline{\nabla}_l\overline{R})_{jlk\overline{i}}+(\overline{\nabla}_j\overline{R})_{kll\overline{i}})+6h_{ijk}h_{ikl}\overline{R}_{jmml}+12h_{ijk}h_{ljm}\overline{R}_{kmli}
    \\&\nonumber\kern1em +6h_{ijk}h_{mkr}h_{mlr}h_{ilj}-4h_{ijk}h_{mkr}h_{mlj}h_{ilr}-2\lambda\eta_1H_lh_{ijk,l}
    \\&\kern1em-12\lambda\eta_1h_{ijk}H_kH_{i,j}+4(3\lambda^2-2)h_{ijk}H_iH_jH_k\Big\}.
\end{align} 
\end{corollary}

Using the technique of Huisken \cite{H1},\cite{H2}, we deduce the argument about the long time existence of $H$-flow.

\begin{lemma}\label{L 2.14}
Let $f_t$ be a solution of the $H$-flow on $[0,T)$, for $0\leq T\leq \infty$. Suppose that $\lambda$ and $|\overline{\nabla}^k\lambda|$ are uniformly bounded on $f_t(S)$ for any positive integer $k$.
If there exists a constant  $\alpha_0$ such that 
\begin{align*}
    \sup_{f_t(S)}|A|^2\leq \alpha_0 \quad \textrm{for $t\in[0,T)$ }.
\end{align*}
Then there exists $\alpha_k$ such that
\begin{align*}
    \sup_{f_t(S)}|\nabla^kA|^2\leq \alpha_k \quad \textrm{for $t\in[0,T)$ }.
\end{align*}
\end{lemma}
\begin{proof}
Since $|\overline{\nabla}^k\lambda|$ is bounded, it follows from (\ref{A}) that 

\begin{align*}
    \frac{\partial}{\partial t}|A|^2\leq \lambda^2\Delta |A|^2+K_1|A|^4+K_2
\end{align*}
and inductively there exists a constant $B(k),C(k)$ depends on the bound $C$ such that 
\begin{align*}
    \frac{\partial}{\partial t}|\nabla^k A|^2\leq \lambda^2\Delta |\nabla^k A|^2+B(k)\sum_{a+b+c=k}|\nabla^a A||\nabla^b A||\nabla^c A||\nabla^k A|+C(k).
\end{align*}
Then, the remaining argument is just the same in the mean curvature flow case (c.f. \cite{H1} and \cite{H2}.).

\end{proof}

\begin{theorem}
Let $f_t$ be a solution of the $H$-flow on $[0,T)$. Suppose that $\lambda$ and $|\overline{\nabla}^k\lambda|$ are uniformly bounded on $f_t(S)$ for any positive integer $k$. If $T<\infty$, then
\begin{align*}
    \lim_{t\to T}\sup_{f_t(S)}|A|^2=\infty.
\end{align*}
\end{theorem}

\begin{proof}
Suppose the theorem is false ,i.e., 
\begin{align*}
    \lim_{t\to T}\sup_{f_t(S)}|A|^2<\infty.
\end{align*}
For any point $p\in S$ and $0\leq t_1<t_2<T$
\begin{align*}
    |f(p,t_2)-f(p,t_1)|&\leq \int_{t_1}^{t_2}|\frac{\partial}{\partial t}f|dt=\int_{t_1}^{t_2}|\lambda\nabla\lambda+\lambda^2 H|dt
    \\&\leq C_0(t_2-t_1).
\end{align*}
Thus, for $t\to T$, $\{f_{t}(S)\}$ converge to a unique continuous limit $f_T(S)$. Moreover, by \Cref{L 2.14}
\begin{align*}
   |\nabla^kf(p,t_2)-\nabla^kf(p,t_1)|&\leq \int_{t_1}^{t_2}|\nabla^k(\lambda\nabla\lambda+\lambda^2 H)|dt
    \\&\leq C_k(t_2-t_1).
\end{align*}
This implies that $f_t(S)$ converges to $f_T(S)$ in the $C^\infty$-topology as $t\to T$. In view of short time existence, we know that there exists a solution for $t=T+\epsilon$ which contradicts the assumption that $T$ is maximal. 

\end{proof}

For the special $H$-flow case, $\lambda$ is bounded due to the evolution formula.  Due to (\ref{Lam}),
the boundedness of the derivative of $\lambda$ can be reduced to the boundedness of the second fundamental form. Thus,
\begin{corollary}
Let $f_t$ be a solution of the $H$-flow on $[0,T)$. If $T<\infty$, then
\begin{align*}
    \lim_{t\to T}\sup_{f_t(S)}|A|^2=\infty.
\end{align*}
\end{corollary}

\section{Strongly Stable Submanifolds}

\subsection{Strongly stable conditions}

In this subsection, we review some properties of the strongly stable condition. Most of the materials can be found in \cite{WT2}. Let $M$ be a $m$-dimensional submanifold in $n$-dimensional manifold $N$ and $\overline{\nabla}$ be the Levi-civita connection on $N$ with respect to its Riemannian metric $\overline{g}$. We first define the following.

\begin{itemize}
    \item The partial Ricci operator is given by $\mathcal{R}(V)=\mathrm{tr}_M(R^N(\cdot,V)\cdot)^\perp$, where $V$ is a normal vector field.
\item The operator $\mathcal{A}:=\mathcal{S}^t\circ\mathcal{S}:TM^\perp\to TM^\perp$, where $\mathcal{S}^t$ is the transpose map of the shape operator $\mathcal{S}$.
\end{itemize}
Recall that the second variation of the area functional is the normal direction $V$ is given by
\begin{align*}
    \int_M|\nabla^\perp V|^2+g(\mathcal{R}(V),V)-\overline{g}(\mathcal{A}(V),V),
\end{align*}
we then give the following definition.

\begin{defn}
A minimal immersed submanifold $M$ in $N$ is called stable if 
\begin{align*}
    \int_M|\nabla^\perp V|^2+\overline{g}(\mathcal{R}(V),V)-\overline{g}(\mathcal{A}(V),V)\geq 0 \quad \textrm{ for any compactly support normal vector $V$.}
\end{align*}
Moreover, $M$ is strongly stable if $\mathcal{R}-\mathcal{A}$ is a pointwise positive operator on normal bundle ,i.e., if there exists $c>0$ such that
\begin{align}
    \sum_{\alpha,\beta,i}R^N_{i\alpha i\beta}V^\alpha V^\beta-\sum_{\alpha,\beta,i,j}h_{\alpha ij}h_{\beta ij}V^\alpha V^\beta\geq c\sum_{\alpha}(V^\alpha)^2, \label{st}
\end{align}
for any normal vector $V=V^\alpha e_{\alpha}$.
\end{defn}

In this paper, we focus on the case when $N$ is a hyperkähler 4-manifold which is equivalent to the Calabi-Yau 4-manifold. Motivated by \cite{WT2}, we look for some examples of Lagrangian surfaces to be examples of strongly stable surfaces. Recall that 
\begin{defn}
A $n$-dimensional submanifold $L$ in a Calabi-Yau manifold $N$ is called a special Lagrangian submanifold if $\omega|_L\equiv 0$ and $\textrm{Im}(\Omega)|_L\equiv 0$, where $\Omega$ is a nowhere vanishing holomorphic form.
\end{defn}

\begin{defn}
A $\frac{n}{2}$-dimensional submanifold $L$ of a hyperk\"{a}hler manifold $N$ is called complex Lagrangian if it is a Lagrangian with respect to some holomorphic symplectic form.
\end{defn}

In \cite{H,H'}, Hitchin proved that
\begin{theorem}
A complex Lagranigan submanifold $L$ in a hyperk\"{a}hler manifold $N$ is a special Lagrangian submanifold. 
\end{theorem}

Note that the special Lagrangain submanifolds are minimal \cite{HL}, we conclude the following. 
\begin{corollary}
  A minimal Lagrangian surface $L$ in a hyperk\"{a}hler 4-manifold is strongly stable if $\Rc^L$ is positive. In particular, a complex Lagrangian surface $L$ in a hyperk\"{a}hler 4-manifold is strongly stable if $\Rc^L$ is positive.
\end{corollary}
\begin{proof}
Suppose $L$ is minimal and $\omega_1\big|_L=0$, we consider the orthonormal basis, $\{\partial_1,\partial_2,I\partial_1,I\partial_2\}$ and compute
\begin{align*}
    R^N(e_i,I(e_k),e_i,I(e_l))&=-R^N(I(e_i),I(e_k),I(e_i),I(e_l))
    \\&=-R^N(e_i,e_k,e_i,e_l)
    \\&=-R^L(e_i,e_k,e_i,e_l)-\langle A(e_i,e_i),A(e_k,e_l) \rangle+\langle A(e_i,e_l),A(e_k,e_i) \rangle,
\end{align*}
where we used the Gauss equation and the fact that $\Rc^N\equiv 0$. Therefore, (\ref{st}) reduces to 
\begin{align*}
    \Rc^L_{kl}V^kV^l.
\end{align*}
\end{proof}

\subsection{Eguchi-Hanson Space}

In this section, we explicitly find a strongly stable surface in a hyperk\"{a}hler 4-manifold. The Eguchi Hanson space is a non-compact, self-dual, asymptotically locally Euclidean (ALE) metric on the cotangent bundle of the 2-sphere $T^*S^2$. This metric is given by physicists Eguchi and Hanson \cite{EH}
\begin{align}
    &g_{EH}=(1-\frac{c}{r^4})^{-1}dr^2+r^2((\sigma^1)^2+(\sigma^2)^2)+r^2(1-\frac{c}{r^4})(\sigma^3)^2, \quad r>\sqrt[4]{c},
\end{align}
where $c$ is a constant, $\sigma^i$ are a left invariant one forms on $SU(2)$ and satisfy $d\sigma^i=2\epsilon^i_{jk}\sigma^j\wedge \sigma^k$. 

\begin{remark}
$r=\sqrt[4]{c}$ is a coordinate singularity. Let $\cosh{u}=\frac{r^2}{\sqrt{c}}$, then
\begin{align*}
    g_{EH}=\frac{\sqrt{c}}{4}\cosh{u}du^2+\sqrt{c}\cosh{u}((\sigma^1)^2+(\sigma^2)^2)+\sqrt{c}\sinh{u}\tanh{u}(\sigma^3)^2
\end{align*}
As $r\to\sqrt[4]{c}$, $g_{EH}\to \sqrt{c}((\sigma^1)^2+(\sigma^2)^2)=\frac{\sqrt{c}}{4}(d \theta^2+\sin^2{\theta}d\varphi^2)$, which is a standard sphere $S^2$ of radius $\frac{\sqrt{c}}{4}$.
\end{remark}

In the following, we use the Cartan's moving frame method to study the  hyperk\"{a}hler structure of Eguchi-Hanson metric. Let
\begin{align*}
    \omega^0=(1-\frac{c}{r^4})^{-\frac{1}{2}}dr, \quad \omega^1=r\sigma_1, \quad \omega^2=r\sigma_2, \quad \omega^3=(1-\frac{c}{r^4})^{\frac{1}{2}}\sigma_3,
\end{align*}
and $\{e_0,e_1,e_2,e_3\}$ be its dual frame. Taking $A=1-\frac{c}{r^4}$, the connection 1-forms are as follows.
\begin{align}
&\nonumber\omega_1^0=\frac{-A^{\frac{1}{2}}}{r}\omega^1, \quad \omega_2^0=\frac{-A^{\frac{1}{2}}}{r}\omega^2, \quad \omega_3^0=\frac{A^{\frac{1}{2}}-2A^{\frac{-1}{2}}}{r}\omega^3,
\\&\omega_2^1=\frac{2A^{\frac{-1}{2}}-A^{\frac{1}{2}}}{r}\omega^3, \quad \omega_3^1=\frac{-A^{\frac{1}{2}}}{r}\omega^2, \quad \omega_3^2=\frac{A^{\frac{1}{2}}}{r}\omega^1,
\end{align}
 and connection 2-forms $R_i^j=d\omega_i^j-\omega_i^k\wedge \omega^j_k$ are
\begin{align}
&\nonumber R_1^0=\frac{2A-2}{r^2}\omega^0\wedge \omega^1+\frac{2-2A}{r^2}\omega^2\wedge \omega^3=-R_3^2,
\\&\nonumber R_2^0=\frac{2A-2}{r^2}\omega^0\wedge \omega^2+\frac{2A-2}{r^2}\omega^1\wedge \omega^3=R_3^1,
\\&R_3^0=\frac{4-4A}{r^2}\omega^0\wedge \omega^3+\frac{4A-4}{r^2}\omega^1\wedge \omega^2=-R_2^1.
\end{align} 
Note that $A\to 0$ when $r\to\sqrt[4]{c}$, the curvature is bounded when $r\to\sqrt[4]{c}$ and this metric is Ricci-flat. We define 3 complex structures with respect to the frame $\{e_0,e_1,e_2,e_3\}$ by
\begin{align*}
I=\begin{pmatrix}
0 & 0 & 0 & -1\\
0 & 0 & -1 & 0 \\
0 & 1 & 0 & 0 \\
1 & 0 & 0 & 0
\end{pmatrix},
\quad
J=\begin{pmatrix}
0 & -1 & 0 & 0\\
1 & 0 & 0 & 0 \\
0 & 0 & 0 & -1 \\
0 & 0 & 1 & 0
\end{pmatrix},
\quad
K=\begin{pmatrix}
0 & 0 & -1 & 0\\
0 & 0 & 0 & 1 \\
1 & 0 & 0 & 0 \\
0 & -1 & 0 & 0
\end{pmatrix}.
\end{align*}
One can check that $I^2=J^2=K^2=-Id$ and $IJ=K$. The corresponding K\"{a}hler forms are 
\begin{align*}
&\omega_I=\omega^0\wedge \omega^3+\omega^1\wedge \omega^2 ,
\\&\omega_J=\omega^0\wedge \omega^1+\omega^2\wedge \omega^3,
\\&\omega_K=\omega^0\wedge \omega^2-\omega^1\wedge \omega^3.
\end{align*}
Throughout this computation, one can easily see that the Eguchi-Hanson metric is hyperk\"{a}hler. Moreover, the zero section $S^2$ is a complex Lagrangian submanifold in the Eguchi-Hanson space.

\begin{proposition}
The zero section $S^2$ of the Eguchi-Hanson space is a totally geodesic surface and is strongly stable.
\end{proposition}
\begin{proof}
In the zero section, the normal vectors are $e_0$ and $e_3$. We compute the second fundamental form.
\begin{align*}
    h_{\alpha ij}=g_{EH}(e_{\alpha},\nabla_{e_i}e_j)=g_{EH}(e_{\alpha},\omega_i^k(e_j)e_k)=\omega_i^\alpha(e_j).
\end{align*}
So the only non-vanishing terms are
\begin{align*}
h_{011}=\frac{-A^{\frac{1}{2}}}{r}, \quad h_{022}=\frac{-A^{\frac{1}{2}}}{r}, \quad h_{312}=\frac{A^{\frac{1}{2}}}{r}.
\end{align*}
All vanish when $r\to\sqrt[4]{c}$. Moreover, we observe that
\begin{align*}
    R_{0110}=R_{0220}=R_{1331}=R_{2332}=\frac{-2}{\sqrt{c}}<0.
\end{align*}
Thus, (\ref{st}) suggests that the zero section is strongly stable.

\end{proof}

\begin{remark}
  In \cite{WT1}, authors also pointed out that any compact, minimal submanifold of Eguchi-Hanson space must be contained in the zero section.  
\end{remark}

\section{The stability of special Hyperkähler flow in strongly stable surface}

In this section, we are going to prove our main result. First, we introduce the tubular neighborhood.

\subsection{Tubular Neighborhood}

\begin{theorem}[Tubular neighborhood theorem]
Let $M$ be a Riemannian manifold and $\Sigma$ be a compact, oriented, embedded  submanifold.
There exists a diffeomorphism from an open neighborhood in normal bundle $N\Sigma$ onto an open neighborhood of $\Sigma$ in $M$.
\end{theorem}

In our case, $M$ is a hyperkähler 4-manifold and $\Sigma$ is a compact complex Lagrangian surface. Given any $p\in \Sigma$, let $U_{\epsilon}$ denote the tubular neighborhood of $p$. For any $q\in U_{\epsilon}\subseteq M$, there exists a unique $p\in\Sigma$ such that $p$ and $q$ are connected by the unique normal geodesic in $U_\epsilon$. By using the parallel transport of $T\Sigma$ along normal geodesic, one can define the horizontal distribution $\mathcal{H}$ and its orthogonal complement $\mathcal{V}$ in $TM$ which is called the vertical distribution. In the following, we denote the local coordinate system by 
$\{x^1,x^2,y^1,y^2\}$ and denote the local frame by $\{e_1,e_2,e_3,e_4\}$. More precisely,
\begin{align*}
    \mathcal{H}=\textrm{span}\{e_1,e_2\}, \quad \mathcal{V}=\textrm{span}\{e_{3},e_{4}\}.
\end{align*}
Then, we can use the parallel transport of the volume form $\Omega$ on $\Sigma$ along the normal geodesic to define a form on $U_\epsilon$. More precisely,
\begin{align}
    \Omega=\omega^1\wedge\omega^2,
\end{align}
where $\{\omega^1,\omega^2,\omega^3,\omega^4\}$ is the dual frame of $\{e_1,e_2,e_3,e_4\}$. Consider $L\subseteq T_qM$ with $\Omega(L)>0$, we can view it as a graph from the horizontal distribution $\mathcal{H}_q$ to the vertical distribution $\mathcal{V}_q$. By singular value decomposition, there exists $\{e_1,e_2\}$ orthonormal basis for $\mathcal{H}_q$ and $\{e_3,e_4\}$ for $\mathcal{V}_q$ such that 
\begin{align}
    \Tilde{e}_i=\cos\theta_i e_i+\sin\theta_i e_{i+2}, \quad  \Tilde{e}_\alpha=-\sin\theta_{\alpha} e_{\alpha-2}+\cos\theta_{\alpha} e_{\alpha},
\end{align}
where $\theta_i\in [0,\frac{\pi}{2})$, $\theta_\alpha=\theta_{\alpha-2}$.

\begin{remark}
Conversely, if we let $\{\Tilde{\omega}^1,\Tilde{\omega}^2,\Tilde{\omega}^3,\Tilde{\omega}^4\}$ be the dual frame of $\{\Tilde{e}_1,\Tilde{e}_2,\Tilde{e}_3,\Tilde{e}_4\}$ then 

\begin{align*}
 &e_i=\cos\theta_i \Tilde{e}_i-\sin\theta_i \Tilde{e}_{i+2}, \quad  e_{\alpha}=\sin\theta_{\alpha} \Tilde{e}_{\alpha-2}+\cos\theta_{\alpha} \Tilde{e}_{\alpha},
 \\&\omega_i=\cos\theta_i \Tilde{\omega}_i-\sin\theta_i \Tilde{\omega}_{i+2}, \quad  \omega_{\alpha}=\sin\theta_{\alpha} \Tilde{\omega}_{\alpha-2}+\cos\theta_{\alpha} \Tilde{\omega}_{\alpha}.
\end{align*}

\end{remark}

Next, we check that the relation between complex structures and this basis.
\begin{lemma}
Let $\Sigma$ be a complex Lagrangian submanifold with $\omega_I+i\omega_K=0$ then 

\begin{itemize}

\item $Ie_1,Ie_2,Ke_1,Ke_2\in \textrm{span}\{e_3,e_4\}$.
\item $Ie_3,Ie_4,Ke_3,Ke_4\in \textrm{span}\{e_1,e_2\}$.

\end{itemize}

\end{lemma}

\begin{proof}
Since $\Sigma$ is a complex Lagrangian submanifold, at any $p\in\Sigma$
\begin{align*}
    0=\omega_K(e_1,e_2)|_p=\overline{g}(Ke_1,e_2)|_p.
\end{align*}
For any $q\in U_\epsilon$, there exists a $p\in\Sigma$ such that $p$ and $q$ are connected by normal geodesic. Since we define our frame by parallel transport,
\begin{align*}
    &\overline{g}(Ke_1,e_2)|_p=\overline{g}(Ke_1,e_2)|_q=-\overline{g}(Ke_2,e_1)|_q=0
    \\&\overline{g}(Ke_1,e_1)|_q=-\overline{g}(Ke_1,e_1)|_q=0
\end{align*}

On the other hand, we check that $Ke_3\in\textrm{span}\{e_1,e_2\}$. Write $Ke_1=a_1 e_3+a_2e_4$ and $Ke_2=b_1 e_3+b_2e_4$, we compute
\begin{align*}
    0&=\overline{g}(e_1,e_3)
    \\&=-\overline{g}(a_1Ke_3+a_2Ke_4,e_3)
    \\&=-a_2 \overline{g}(Ke_4,e_3).
\end{align*}
Similarly, we get 
\begin{align*}
    0&=a_1\overline{g}(Ke_3,e_4)=b_1\overline{g}(Ke_3,e_4)
    \\&=a_2\overline{g}(Ke_4,e_3)=-a_2\overline{g}(Ke_3,e_4)
    \\&=b_2\overline{g}(Ke_4,e_3)=-b_2\overline{g}(Ke_3,e_4).
\end{align*}
Since $a_1,a_2,b_1,b_2$ not all vanish, it implies that $\overline{g}(Ke_3,e_4)=0$. Others are similar.

\end{proof}

At the end of this subsection, we recall some estimates in \cite{WT2} which we will use later. 
 Let $U_\epsilon$ be the tubular neighborhood of $p\in \Sigma$ with the coordinate system $(x^1,x^2,y^{3},y^4)$ and the local frame $\{e_1,e_2,e_3,e_4\}$ with the dual frame $\{\omega^1,\omega^2,\omega^3,\omega^4\}$. Then we have

\begin{lemma}[\cite{WT2}, Lemma 2.5] \label{L4.4}

\begin{align*}
    &\nonumber \overline{g}(\frac{\partial}{\partial x^i},e_j)|_{(x,y)}=\delta_{ij}-y^\alpha h_{\alpha ij}|_p+O(|x|^2+|y|^2),
    \\&\nonumber\overline{g}(\frac{\partial}{\partial y^\alpha},e_\beta)|_{(x,y)}=\delta_{\alpha\beta}+O(|x|^2+|y|^2),
    \\&\overline{g}(\frac{\partial}{\partial x^i},e_\beta)|_{(x,y)}=O(|x|^2+|y|^2) \quad \overline{g}( \frac{\partial}{\partial y^\alpha},e_j)|_{(x,y)}=O(|x|^2+|y|^2).
\end{align*}
In summary, we can write
\begin{align*}
    &\nonumber e_i=\frac{\partial}{\partial x^i}+y^\alpha h_{\alpha ij}|_p\frac{\partial}{\partial x^j}+O(|x|^2+|y|^2),  \\&e_\alpha=\frac{\partial}{\partial y^\alpha}+O(|x|^2+|y|^2).
\end{align*}

\end{lemma}

\begin{lemma}[\cite{WT2}, Proposition 2.6]\label{L4.5}
For the connection 1-form $\omega_A^B=\langle \overline{\nabla}_{e_C}e_A,e_B \rangle \omega^C$, we have the following expansion. 
\begin{align*}
    &\nonumber\omega_i^j(e_k)|_{(x,y)}=\frac{1}{2}x^lR_{jikl}|_p+y^\alpha \overline{R}_{jik\alpha }|_p+O(|x|^2+|y|^2),
    \\&\nonumber\omega_i^j(e_\beta)|_{(x,y)}=\frac{1}{2}y^\alpha \overline{R}_{ji\beta\alpha}|_p+O(|x|^2+|y|^2),
    \\&\nonumber\omega_i^\alpha(e_j)|_{(x,y)}=h_{\alpha ij}|_p+x^kh_{\alpha ij,k}|_p+y^\beta(\overline{R}_{\alpha ij \beta}+h_{\alpha il}h_{\beta jl})|_p+O(|x|^2+|y|^2),
    \\&\nonumber\omega_i^\alpha(e_\beta)|_{(x,y)}=\frac{1}{2}y^\gamma \overline{R}_{\alpha i\beta\gamma}|_p+O(|x|^2+|y|^2),
    \\&\nonumber\omega_\beta^\alpha(e_i)|_{(x,y)}=\frac{1}{2}x^j R^\perp_{\alpha\beta ij}|_p+y^\gamma \overline{R}_{\alpha\beta i \gamma}+O(|x|^2+|y|^2),
    \\&\omega_\beta^\alpha(e_\gamma)|_{(x,y)}=\frac{1}{2}y^\delta \overline{R}_{\alpha\beta\gamma\delta}|_p+O(|x|^2+|y|^2),
\end{align*}
where $\overline{R}$ deonte the curvature on $M$, $R$ is the curvature on $\Sigma$ and $R^\perp$ is the normal curvature. 
\end{lemma}

\subsection{Previous Estimate}
Let $\Gamma\subset U_\epsilon$ be an oriented surface. For any $q\in\Gamma$, we construct a basis $\{\Tilde{e}_1,\Tilde{e}_2\}$ for its tangent space $T_q\Gamma$ by using the method in the last subsection. In the following, we define some tensors that will be used in the proof of the main theorem.

\begin{itemize}

\item $s=\max\{\sin\theta_1,\sin\theta_2\}$.

\item Denote the second fundamental forms are given by
\begin{align*}
  &\mathrm{II}^{\Gamma}=\Tilde{h}_{\alpha ij} \Tilde{e}_\alpha\otimes \Tilde{\omega}^i\otimes \Tilde{\omega}^j,
  \\&\mathrm{II}^{\Sigma}=h_{\alpha ij} e_\alpha\otimes \omega^i\otimes \omega^j=h_{\alpha ij}(\sin\theta_\alpha \Tilde{e}_{\alpha-2}+\cos\theta_\alpha \Tilde{e}_{\alpha})\otimes (\cos\theta_i\Tilde{\omega}^i-\sin\theta_i \Tilde{\omega}^{i+2})\otimes (\cos\theta_j\Tilde{\omega}^j-\sin\theta_j \Tilde{\omega}^{j+2}),
\end{align*}
then
\begin{align*}
    \langle \mathrm{II}^{\Gamma},\mathrm{II}^{\Sigma}\rangle|_q=\sum_{\alpha,i,j}\cos\theta_i\cos\theta_j\cos\theta_\alpha  \Tilde{h}_{\alpha ij}h_{\alpha ij}.
\end{align*}

\item Define the tensor 
\begin{align*}
    S^\Sigma|_q=y^\beta(\overline{R}_{\alpha ij\beta}+h_{\alpha il}h_{\beta jl})|_p \omega^i\otimes \omega^j \otimes e_\alpha,
\end{align*}
where $p\in\Sigma$ is the unique point such that there exists a unique normal geodesic connecting $p$ and $q$. Then,
\begin{align*}
    \langle \mathrm{II}^{\Gamma},S^{\Sigma}\rangle|_q=\sum_{\alpha,\beta,i,j}\cos\theta_i\cos\theta_j\cos\theta_\alpha  \Tilde{h}_{\alpha ij}y^\beta(\overline{R}_{\alpha ij\beta}+h_{\alpha il}h_{\beta jl})|_p.
\end{align*}

\end{itemize}

Before proving the main result, we need to use above two tensors to obtain some estimates of $\Omega$ in the tubular neighborhood $U_\epsilon$. The general results for the following estimate can be found in \cite{WT2}. In the following, we suppose $\Omega(T_q\Gamma)>\frac{1}{2}$.

\begin{lemma}

\begin{align*}
    &\Big|\langle \mathrm{II}^{\Gamma},\mathrm{II}^{\Sigma}\rangle|_q-\sum_{\alpha,i,j} \Tilde{h}_{\alpha ij}h_{\alpha ij}\Big|\leq cs^2|\mathrm{II}^{\Gamma}|,
    \\& \Big|\langle \mathrm{II}^{\Gamma},S^{\Sigma}\rangle|_q-\sum_{\alpha,\beta,i,j}\Tilde{h}_{\alpha ij}y^\beta(\overline{R}_{\alpha ij\beta}+h_{\alpha il}h_{\beta jl})\big|_p\Big|\leq cs^2 \sqrt{\psi}|\mathrm{II}^{\Gamma}|,
\end{align*}
where $\psi=\sum_\alpha (y^\alpha)^2 $ is the square of the distance to $\Sigma$.
\end{lemma}
\begin{proof}
It suffices to estimate 
\begin{align*}
    |1-\cos^2\theta_i\cos\theta_j|&=|1-\cos\theta_j+\sin^2\theta_i\cos\theta_j|
    \\&\leq |1-\sqrt{1-\sin^2\theta_j}|+s^2 \leq \frac{3}{2}s^2. 
\end{align*}

\end{proof}

\begin{lemma}
\begin{align*}
    \Big|\sum_{\alpha=3,4}\sum_{k=1,2}\Tilde{\Omega}_{\alpha 2,k}\Tilde{h}_{\alpha k1}+\Tilde{\Omega}_{1\alpha, k}\Tilde{h}_{\alpha k2}-\langle\mathrm{II}^{\Gamma},\mathrm{II}^{\Sigma}+S^\Sigma \rangle\big|_q(*\Omega)\Big|\leq c(s^2+\psi)|\mathrm{II}^{\Gamma}|.
\end{align*}
Here, $*\Omega=\Omega(\tilde{e}_1,\tilde{e}_2)$, $\Tilde{\Omega}_{AB}=\Omega(\Tilde{e}_A,\Tilde{e}_B)$, $\tilde{\Omega}_{\alpha 2,k}=(\nabla_{\tilde{e}_k}\Omega)(\tilde{e}_\alpha,\tilde{e}_2)$.
\end{lemma}
\begin{proof}
We compute that
\begin{align*}
   \overline{\nabla}_{\Tilde{e}_k}\Omega=(\overline{\nabla}_{\tilde{e}_k}\omega_1)\wedge\omega_2+\omega_1\wedge(\overline{\nabla}_{\tilde{e}_k}\omega_2)=-\omega_\beta^1(\Tilde{e}_k)\omega^\beta\wedge\omega^2-\omega_\beta^2(\Tilde{e}_k)\omega^1\wedge\omega^\beta.
\end{align*}
Here,
\begin{align*}
    &-\sum_{\alpha,\beta}\omega_\beta^1(\Tilde{e}_k)(\omega^\beta\wedge\omega^2)(\Tilde{e}_\alpha,\Tilde{e}_2)\Tilde{h}_{\alpha k1}=-\omega_3^1(\Tilde{e}_k)\Tilde{h}_{3k1}\cos\theta_1\cos\theta_2-\omega_4^1(\Tilde{e}_k)\Tilde{h}_{4k1}
    \\&-\sum_{\alpha,\beta}\omega_\beta^2(\Tilde{e}_k)(\omega^1\wedge\omega^\beta)(\Tilde{e}_\alpha,\Tilde{e}_2)\Tilde{h}_{\alpha k1}=\omega_4^2(\Tilde{e}_k)\Tilde{h}_{3k1}\sin\theta_1\sin\theta_2
    \\&-\sum_{\alpha,\beta}\omega_\beta^1(\Tilde{e}_k)(\omega^\beta\wedge\omega^2)(\Tilde{e}_1,\Tilde{e}_\alpha)\Tilde{h}_{\alpha k2}=\omega_3^1(\Tilde{e}_k)\Tilde{h}_{4k2}\sin\theta_1\sin\theta_2
    \\&-\sum_{\alpha,\beta}\omega_\beta^2(\Tilde{e}_k)(\omega^1\wedge\omega^\beta)(\Tilde{e}_1,\Tilde{e}_\alpha)\Tilde{h}_{\alpha k2}=-\omega_3^2(\Tilde{e}_k)\Tilde{h}_{3k2}-\omega_4^2(\Tilde{e}_k)\Tilde{h}_{4k2}\cos\theta_1\cos\theta_2.
\end{align*}   
All terms related to $\sin$ are bounded by $s$ and connection 1-forms are bounded. So we get
\begin{align*}
   &\Big|\sum_{\alpha=3,4}\sum_{k=1,2}\Tilde{\Omega}_{\alpha 2,k}\Tilde{h}_{\alpha k1}+\Tilde{\Omega}_{1\alpha, k}\Tilde{h}_{\alpha k2}\Big|
   \\&\kern1em\leq \Big|\omega_1^3(\Tilde{e}_k)\Tilde{h}_{3k1}\cos\theta_1\cos\theta_2+\omega_1^4(\Tilde{e}_k)\Tilde{h}_{4k1}\cos^2\theta_2+\omega_2^3(\Tilde{e}_k)\Tilde{h}_{3k2}\cos^2\theta_1+\omega_2^4(\Tilde{e}_k)\Tilde{h}_{4k2}\cos\theta_1\cos\theta_2\Big|+c|\mathrm{II}^{\Gamma}|s^2.
\end{align*}
By \Cref{L4.5}, connection 1-forms in the tubular neighborhood satisfy
\begin{align*}
    \omega^\alpha_i(\Tilde{e}_k)|_q&=\cos\theta_k\omega^\alpha_i(e_k)+\sin\theta_k\omega^\alpha_i(e_{k+2})
    \\&=\cos\theta_k\big[h_{\alpha ik}|_p+y^\beta(\overline{R}_{\alpha i k\beta}+h_{\alpha il}h_{\beta kl})|_p\big]+\sin\theta_k(\frac{1}{2}y^\gamma \overline{R}_{\alpha i(k+2)\gamma })+O(|y|^2)
\end{align*}
then 
\begin{align*}
    \Big|\omega^\alpha_i(\Tilde{e}_k)-\cos\theta_k\big[h_{\alpha ik}|_p+y^\beta(\overline{R}_{\alpha ik\beta }+h_{\alpha il}h_{\beta kl})|_p\big]\Big|\leq c(s^2+\psi).
\end{align*}
Thus, we have two terms 

\begin{enumerate}
    \item [(a)] \begin{align*}
    &\sum_{\alpha ik}h_{\alpha ki}\Tilde{h}_{\alpha ki}\cos\theta_k\frac{\cos\theta_\alpha}{\cos\theta_i}(*\Omega),
\end{align*}

\item [(b)] \begin{align*}
    \sum_{\alpha ik}\Tilde{h}_{\alpha ki}\cos\theta_k\frac{\cos\theta_\alpha}{\cos\theta_i}y^\beta(\overline{R}_{\alpha ik\beta }+h_{\alpha il} h_{\beta kl})(*\Omega).
\end{align*}

\end{enumerate}
Note that
\begin{align*}
    |\frac{1}{\cos\theta_i}-\cos\theta_i|=|\frac{\sin^2\theta_i}{\cos\theta_i}|\leq 2s^2,
\end{align*}
so
\begin{align*}
    &\Big|(a)-\langle \mathrm{II}^{\Gamma},\mathrm{II}^{\Sigma}\rangle|_q(*\Omega)\Big|\leq c |\mathrm{II}^{\Gamma}|s^2
    \\&\Big|(b)-\langle \mathrm{II}^{\Gamma},S^{\Sigma}\rangle|_q(*\Omega)\Big|\leq  c |\mathrm{II}^{\Gamma}|\sqrt{\psi}s^2
\end{align*}

\end{proof}

\begin{lemma} \label{L4.8}

\begin{align*}
   \Big|(\nabla_{\Tilde{e}_k}\Omega)(\Tilde{e}_1,\Tilde{e}_2)\Big|\leq c(s^2+\psi)+cs
\end{align*}
and 
\begin{align*}
    \big|\nabla(*\Omega)\big|^2\leq cs^2 |\mathrm{II}^\Gamma-\mathrm{II}^\Sigma|^2+c(s^2+\psi)^2.
\end{align*}

\end{lemma}
\begin{proof}
First, we get
\begin{align*}
   (\nabla_{\Tilde{e}_k}\Omega)(\Tilde{e}_1,\Tilde{e}_2)&=\omega^\beta_1(\Tilde{e}_k)\omega^\beta\wedge\omega^2(\Tilde{e}_1,\Tilde{e}_2)+\omega^\beta_2(\Tilde{e}_k)\omega^1\wedge\omega^\beta(\Tilde{e}_1,\Tilde{e}_2)
   \\&=\omega^3_1(\Tilde{e}_k)\sin\theta_1\cos\theta_2+\omega^4_2(\Tilde{e}_k)\cos\theta_1\sin\theta_2
   \\&=\sum_{i} \omega_i^{i+2}(\tilde{e}_k)\frac{\sin\theta_i}{\cos\theta_i}(*\Omega).
\end{align*}
Due to \Cref{L4.5},
we know that
\begin{align*}
    &\Big|(\nabla_{\Tilde{e}_k}\Omega)(\Tilde{e}_1,\Tilde{e}_2)-\sum_i \cos\theta_k h_{(i+2)ik}\frac{\sin\theta_i}{\cos\theta_i}(*\Omega)-\sum_i\cos\theta_ky^\beta(\overline{R}_{(i+2) ik\beta }+h_{(i+2) il}h_{\beta kl})\frac{\sin\theta_i}{\cos\theta_i}(*\Omega)\Big|\leq c(s^2+\psi)
    \\&\Longrightarrow \Big|(\nabla_{\Tilde{e}_k}\Omega)(\Tilde{e}_1,\Tilde{e}_2)-\sum_i \cos\theta_k h_{(i+2)ik}\frac{\sin\theta_i}{\cos\theta_i}(*\Omega)\Big|\leq c(s^2+\psi)
    \\&\Longrightarrow \Big|(\nabla_{\Tilde{e}_k}\Omega)(\Tilde{e}_1,\Tilde{e}_2)-\sum_i \cos\theta_k\cos^2\theta_i h_{(i+2)ik}\frac{\sin\theta_i}{\cos\theta_i}(*\Omega)\Big|\leq c(s^2+\psi)
\end{align*}
Next,  
\begin{align*}
    \nabla(*\Omega)&=\Tilde{e}_k(\Omega(\Tilde{e}_1,\Tilde{e}_2))\Tilde{\omega}^k
    \\&=[\nabla_{\Tilde{e}_k}\Omega)(\Tilde{e}_1,\Tilde{e}_2)+\Omega(\nabla_{\Tilde{e}_k}\Tilde{e}_1,\Tilde{e}_2)+\Omega(\Tilde{e}_1,\nabla_{\Tilde{e}_k}\Tilde{e}_2)]\Tilde{\omega}^k
    \\&=[(\nabla_{\Tilde{e}_k}\Omega)(\Tilde{e}_1,\Tilde{e}_2)+\Tilde{h}_{\alpha k1}\Omega(\Tilde{e}_\alpha,\Tilde{e}_2)+\Tilde{h}_{\alpha k2}\Omega(\Tilde{e}_1,\Tilde{e}_\alpha)]\Tilde{\omega}^k.
\end{align*}
and
\begin{align*}
    \Tilde{h}_{\alpha k1}\Omega(\Tilde{e}_\alpha,\Tilde{e}_2)+\Tilde{h}_{\alpha k2}\Omega(\Tilde{e}_1,\Tilde{e}_\alpha)&=-\tilde{h}_{3k1}\cos\theta_2\sin\theta_1-\tilde{h}_{4k2}\cos\theta_1\sin\theta_2
    \\&=-\sum_i(\frac{\sin\theta_i}{\cos\theta_i}\tilde{h}_{(i+2)ik})(*\Omega).
\end{align*}
Thus,
\begin{align*}
    |\nabla(*\Omega)|^2&=|\Tilde{e}_k(\Omega(\Tilde{e}_1,\Tilde{e}_2))|^2
    \\&\leq |(\nabla_{\Tilde{e}_k}\Omega)(\Tilde{e}_1,\Tilde{e}_2)-\cos\theta_k\cos^2\theta_i h_{(i+2)ik}\frac{\sin\theta_i}{\cos\theta_i}(*\Omega)|^2+(*\Omega)^2|\frac{\sin\theta_i}{\cos\theta_i}( \cos\theta_k\cos^2\theta_i h_{(i+2)ik}-\tilde{h}_{(i+2)ik})|^2
    \\&\leq c(s^2+\psi)^2+cs^2(*\Omega)^2|\frac{\sin\theta_i}{\cos\theta_i}( \cos\theta_k\cos^2\theta_i h_{(i+2)ik}-\tilde{h}_{(i+2)ik})|^2
\end{align*}
We observe that
\begin{align*}
    |\mathrm{II}^\Gamma-\mathrm{II}^\Sigma|^2\geq \sum_{\alpha ij}|\tilde{h}_{\alpha ij}-\cos\theta_i\cos\theta_j\cos\theta_\alpha h_{\alpha ij}|^2
\end{align*}
Then we complete the proof.

\end{proof}

\begin{lemma} \label{L4.9}

\begin{align*}
    \Big|(\nabla^2_{\tilde{e}_k,\tilde{e}_k}{\Omega})(\tilde{e}_1,\tilde{e}_2)-(\Tilde{\Omega}_{32}\overline{R}_{\tilde{1}\tilde{k}\tilde{k}\tilde{3}}+\Tilde{\Omega}_{42}\overline{R}_{\tilde{1}\tilde{k}\tilde{k}\tilde{4}}+\Tilde{\Omega}_{13}\overline{R}_{\tilde{2}\tilde{k}\tilde{k}\tilde{3}}+\Tilde{\Omega}_{14}\overline{R}_{\tilde{2}\tilde{k}\tilde{k}\tilde{4}})+|\mathrm{II}^\Sigma+S^\Sigma|^2(*\Omega)\Big|\leq c(s^2+\psi).
\end{align*}
Here, $\overline{R}_{\tilde{1}\tilde{k}\tilde{k}\tilde{3}}=\overline{g}(\overline{R}(\tilde{e}_1,\tilde{e}_k)\tilde{e}_k,\tilde{e}_3)$. Others are similar.
\end{lemma}  
\begin{proof}
\begin{align*}
    &\nabla\Omega=\nabla\omega^1\wedge\omega^2+\omega^1\wedge \nabla\omega^2=(\omega^\alpha_1\otimes\omega^\alpha)\wedge \omega^2+\omega^1\wedge(\omega^\alpha_2\otimes\omega^\alpha)=\omega^\alpha_1\otimes(\omega^\alpha\wedge\omega^2)+\omega^\alpha_2\otimes(\omega^1\wedge\omega^\alpha)
\end{align*}
then
\begin{align*}
    \nabla^2\Omega&=-(\omega^\alpha_1\otimes\omega^\alpha_1+\omega^\alpha_2\otimes\omega^\alpha_2)\otimes\Omega+(\nabla\omega^\alpha_1+\omega^\beta_1\otimes\omega^\alpha_\beta+\omega_2^\alpha\otimes\omega^1_2)\otimes(\omega^\alpha\wedge\omega^2)
    \\&\kern1em+(\nabla\omega^\alpha_2+\omega^\beta_2\otimes\omega^\alpha_\beta+\omega_1^\alpha\otimes\omega^2_1)\otimes(\omega^1\wedge\omega^\alpha)+2(\omega_1^3\otimes\omega_2^4-\omega_1^4\otimes\omega_2^3)\otimes(\omega^3\wedge\omega^4).
\end{align*}   
We estimate
\begin{align}\label{17}
  &\nonumber\Big|\sum_{\alpha,i,k}(\omega^\alpha_i(\tilde{e}_k))^2-\sum_{\alpha,i,k}\cos^2\theta_k[h_{\alpha ik}+y^\beta(\overline{R}_{\alpha i k\beta}+h_{\alpha il}h_{\beta kl})]^2\Big|\leq c(s^2+\psi)
  \\& \nonumber\Longrightarrow \Big|\sum_{\alpha,i,k}(\omega^\alpha_i(\tilde{e}_k))^2-\sum_{\alpha,i,k}[h_{\alpha ik}+y^\beta(\overline{R}_{\alpha i k\beta}+h_{\alpha il}h_{\beta kl})]^2\Big|\leq c(s^2+\psi)
  \\& \Longrightarrow \Big|\sum_{\alpha,i,k}(\omega^\alpha_i(\tilde{e}_k))^2-|\mathrm{II}^\Sigma+S^\Sigma|^2\Big|\leq c(s^2+\psi).
\end{align}
Note that
\begin{align*}
    \nabla\omega_i^\alpha=d\omega_i^\alpha(e_A)\otimes\omega^A+\omega_i^\alpha(e_A)\nabla\omega^A.
\end{align*}
By a direct computation and using \Cref{L4.4} and \Cref{L4.5} , we derive
\begin{align*}
    \big|d\omega_i^\alpha(e_A)(\tilde{e}_k)\omega^A(\tilde{e}_k)-\cos^2\theta_kh_{\alpha ik,k}\big|\leq c(s+\sqrt{\psi})
\end{align*}
and 
\begin{align*}
    \big|\omega_i^\alpha(e_A)\nabla\omega^A(\tilde{e}_k,\tilde{e}_k)\big|\leq c(s+\sqrt{\psi}).
\end{align*}
Therefore
\begin{align*}
    \big|(\nabla\omega_i^\alpha)(\tilde{e}_k,\tilde{e}_k)-\cos^2\theta_kh_{\alpha ik,k}\big|\leq c(s+\sqrt{\psi}).
\end{align*}
Next, one can see that
\begin{align*}
\big|(\omega^\beta_1\otimes\omega^\alpha_\beta+\omega_2^\alpha\otimes\omega^1_2)(\tilde{e}_k,\tilde{e}_k)\big|\leq c(s+\sqrt{\psi}).
\end{align*}
Then,
\begin{align}\label{18}
    &\nonumber \Big|(\nabla\omega^\alpha_1+\omega^\beta_1\otimes\omega^\alpha_\beta+\omega_2^\alpha\otimes\omega^1_2)(\tilde{e}_k,\tilde{e}_k)(\omega^\alpha\wedge\omega^2)(\tilde{e}_1,\tilde{e}_2)-h_{31k,k}\sin\theta_1\cos\theta_2\Big|\leq c(s^2+\psi), 
    \\&  \Big|(\nabla\omega^\alpha_2+\omega^\beta_2\otimes\omega^\alpha_\beta+\omega_1^\alpha\otimes\omega^2_1)(\tilde{e}_k,\tilde{e}_k)(\omega^1\wedge\omega^\alpha)(\tilde{e}_1,\tilde{e}_2)-h_{42k,k}\sin\theta_2\cos\theta_1\Big|\leq c(s^2+\psi) .
\end{align}
Combine (\ref{17}) and (\ref{18}), we derive that
\begin{align*}
    \Big|\nabla^2_{\tilde{e}_k,\tilde{e}_k}\Omega(\tilde{e}_1,\tilde{e}_2)+|\mathrm{II}^\Sigma+S^\Sigma|^2(*\Omega)-\frac{\sin\theta_i}{\cos\theta_i} h_{(2+i)ik,k}(*\Omega)\Big|\leq c(s^2+\psi). 
\end{align*}
Finally, we see that 
\begin{align*}
    \Tilde{\Omega}_{32}\overline{R}_{\tilde{1}\tilde{k}\tilde{k}\tilde{3}}+\Tilde{\Omega}_{42}\overline{R}_{\tilde{1}\tilde{k}\tilde{k}\tilde{4}}+\Tilde{\Omega}_{13}\overline{R}_{\tilde{2}\tilde{k}\tilde{k}\tilde{3}}+\Tilde{\Omega}_{14}\overline{R}_{\tilde{2}\tilde{k}\tilde{k}\tilde{4}}&=-\cos\theta_2\sin\theta_1\overline{R}_{\tilde{1}\tilde{k}\tilde{k}\tilde{3}}-\cos\theta_1\sin\theta_2\overline{R}_{\tilde{2}\tilde{k}\tilde{k}\tilde{4}}
    \\&=-\frac{\sin\theta_i}{\cos\theta_i}(*\Omega)\overline{R}_{\tilde{i}\tilde{k}\tilde{k}\tilde{(i+2)}}
\end{align*}
Then
\begin{align*}
\big|\frac{\sin\theta_i}{\cos\theta_i}(*\Omega)\overline{R}_{\tilde{i}\tilde{k}\tilde{k}\tilde{(i+2)}}-\frac{\sin\theta_i}{\cos\theta_i}(*\Omega)\overline{R}_{ikk(i+2)}\big|\leq c(s^2+\psi).
\end{align*}
By Codzzi equation $\overline{R}_{(2+i)kki}=-h_{(2+i)ik,k}$, we have   
\begin{align*}
     &\Big|(\nabla^2_{\tilde{e}_k,\tilde{e}_k}{\Omega})(\tilde{e}_1,\tilde{e}_2)-(\Tilde{\Omega}_{32}\overline{R}_{\tilde{1}\tilde{k}\tilde{k}\tilde{3}}+\Tilde{\Omega}_{42}\overline{R}_{\tilde{1}\tilde{k}\tilde{k}\tilde{4}}+\Tilde{\Omega}_{13}\overline{R}_{\tilde{2}\tilde{k}\tilde{k}\tilde{3}}+\Tilde{\Omega}_{14}\overline{R}_{\tilde{2}\tilde{k}\tilde{k}\tilde{4}})+|\mathrm{II}^\Sigma+S^\Sigma|^2(*\Omega)\Big|
     \\&=\Big|(\overline{\nabla}^2_{\tilde{e}_k,\tilde{e}_k}{\Omega})(\tilde{e}_1,\tilde{e}_2)+\frac{\sin\theta_i}{\cos\theta_i}(*\Omega)\overline{R}_{\tilde{i}\tilde{k}\tilde{k}\tilde{(i+2)}}+|\mathrm{II}^\Sigma+S^\Sigma|^2(*\Omega)\Big|
     \\&\leq \Big|(\overline{\nabla}^2_{\tilde{e}_k,\tilde{e}_k}{\Omega})(\tilde{e}_1,\tilde{e}_2)-\frac{\sin\theta_i}{\cos\theta_i}(*\Omega)h_{(2+i)ik,k}+|\mathrm{II}^\Sigma+S^\Sigma|^2(*\Omega)\Big|+c(s^2+\psi)
     \\&\leq c(s^2+\psi).
\end{align*}

\end{proof}

\subsection{$C^0$ and $C^1$ Estimate}
In this section, we begin to prove the stability of special hyperk\"{a}hler flow. Let $(M,\overline{g},\overline{\nabla},I,J,K)$ be a hyperkähler $4$-manifold, $\Sigma\subset M$ be a compact, oriented, strongly stable complex Lagrangian surface with respect to $I,K$ ,i.e., $\omega_I+i\omega_K\equiv 0$ on $\Sigma$. Let $\Gamma\subset U_\epsilon$ be a Lagrangian surface that is $C^1$ close to $\Sigma$
in the sense that 
\begin{align}
    \sup_{q\in\Gamma}\big(1-(*\Omega)+K\psi\big)<\kappa \quad \textrm{ for some constant $0<\kappa<<1$ and $K>0$.} 
\end{align}
Consider the special hyperk\"{a}hler flow $\Gamma^t$ with $\Gamma^0=\Gamma$ and write the second fundamental form by $\mathrm{II}^t$. First, we show that $\Gamma^t$ remain in the tubular neighborhood $U_\epsilon$.

\begin{lemma}\label{L4.10}
There exists a tubular neighborhood $U_\epsilon$ such that for any $q\in U_\epsilon$ and any oriented $2$-plane $L\subset T_qM$, we have
\begin{align}
    \textrm{tr}_L\textrm{Hess}(\psi)\geq c(s^2+\psi(q)) \quad \textrm{ for some $c>0$.}
\end{align}
\end{lemma}
\begin{proof}
For any $q\in U_\epsilon$, let $p\in\Sigma$ connect to $q$ by a normal geodesic. Consider the geodesic distance from the zero section $\psi=\sum_\alpha (y^\alpha)^2$. Then,
\begin{align*}
    \textrm{Hess}(\psi)(e_A,e_B)=e_A(e_B(\psi))-(\overline{\nabla}_{e_A}e_B)(\psi).
\end{align*}
By \Cref{L4.4} and \Cref{L4.5}, we compute
\begin{align*}
     &\textrm{Hess}(\psi)(e_i,e_j)=-2y^\alpha \omega_j^\alpha(e_i)=-2y^\alpha\Big[h_{\alpha ji}|_p+x^kh_{\alpha ji,k}|_p+y^\beta(\overline{R}_{\alpha ji \beta}+h_{\alpha jl}h_{\beta il})|_p+O(|x|^2+|y|^2)\Big],
    \\&\textrm{Hess}(\psi)(e_\alpha,e_i)=-2y^\beta \omega_i^\beta(e_\alpha)=-2y^\beta\Big[\frac{1}{2}y^\gamma \overline{R}_{\beta i\alpha\gamma}|_p+O(|x|^2+|y|^2)\Big],
\\&    \textrm{Hess}(\psi)(e_\alpha,e_\beta)=2e_\alpha(y^\beta)-2y^\gamma \omega_\beta^\gamma(e_\alpha)=2\delta_{\alpha\beta}-2y^\gamma\Big[\frac{1}{2}y^\delta \overline{R}_{\gamma\beta\alpha\delta}|_p+O(|x|^2+|y|^2)\Big].
\end{align*}
Note that $x^k$ component vanish,
\begin{align*}
    \textrm{tr}_L\textrm{Hess}(\psi)&=\textrm{Hess}(\psi)(\tilde{e_1},\tilde{e}_1)+\textrm{Hess}(\psi)(\tilde{e_2},\tilde{e}_2)
    \\&=\sum_i\cos^2\theta_i\textrm{Hess}(\psi)(e_i,e_i)+2\cos\theta_i\sin\theta_i\textrm{Hess}(\psi)(e_i,e_{i+2})+\sin^2\theta_i\textrm{Hess}(\psi)(e_{i+2},e_{i+2})
    \\&=\sum_i-2\cos^2\theta_iy^\alpha y^\beta(\overline{R}_{\alpha ii \beta}+h_{\alpha il}h_{\beta il})+2\sin^2\theta_i+s(L)\cdot O(|y|^2)+O(|y|^3)
    \\&\geq \sum_i2c_0\cos^2\theta_i|y|^2+2\sin^2\theta_i-s^2-c|y|^3
    \\&\geq (4c_0-c|y|)|y|^2+\sum_i(2-2c_0|y|^2)\sin^2\theta_i-s^2
    \\&\geq c'(s^2+\psi)
\end{align*}
Here, we use the strongly stable condition (\ref{st}).

\end{proof}

\begin{proposition} \label{P4.11}
There exists $\epsilon>0$ such that the square distance $\psi$ to the minimal strongly stable surface $\Sigma$ satisfies
 \begin{align*}
\frac{\partial}{\partial t}(\psi)\leq \lambda^2(\Delta^{\Gamma_t}(\psi)-c_1(s^2+\psi))+\langle \nabla(\frac{\lambda^2}{2}),\nabla\psi \rangle 
\end{align*}
for some $c_1>0$ in the tubular neighborhood $U_\epsilon$. Thus, $\psi$ is non-increasing by the maximum principle.  Moreover, there exists $c_2>0$ such that
\begin{align}\label{c0}
    \frac{\partial}{\partial t}(\psi)\leq \lambda^2(\Delta^{\Gamma_t}(\psi)-c_1(s^2+\psi)+c_2s^2\sqrt{\psi}|\mathrm{II}^t|).
\end{align}
\end{proposition}
\begin{proof}
We compute
\begin{align*}
\frac{\partial}{\partial t}(\psi)&=\overline{\nabla}_{\lambda\nabla\lambda+\lambda^2 H}\psi=\overline{\nabla}_{\lambda\nabla\lambda}\psi+\overline{\nabla}_{\lambda^2H}\psi
\\&=\lambda^2(\Delta^{f_t(S)}\psi-\mathrm{tr}_{f_t(S)}\mathrm{Hess}(\psi))+\overline{\nabla}_{\lambda\nabla\lambda}\psi.
\end{align*}  
By \Cref{L4.10}, we know that there exists $\epsilon>0$ and $c_1>0$ such that 
\begin{align*}
    \mathrm{tr}_{f_t(S)}\mathrm{Hess}(\psi)\geq c_1(s^2+\psi).
\end{align*}
Also, 
\begin{align*}
    d\psi(\lambda\nabla\lambda)=2\lambda y_\alpha\omega_\alpha(\nabla \lambda)=2\lambda\big[y_3\Tilde{e}_1(\lambda)\sin\theta_1+y_4\Tilde{e}_2(\lambda)\sin\theta_2\big].
\end{align*}
Then,
\begin{align} \label{ei}
    \tilde{e}_i(\lambda)\nonumber&=-\lambda^2(\omega_J(\overline{\nabla}_{\tilde{e}_i}\tilde{e}_1,\tilde{e}_2)+\omega_J(\tilde{e}_1,\overline{\nabla}_{\tilde{e}_i}\tilde{e}_2))
    \\&\nonumber=-\lambda^2\big[\tilde{h}_{3i1}\omega_J(\tilde{e}_3,\tilde{e}_2)+\tilde{h}_{4i1}\omega_J(\tilde{e}_4,\tilde{e}_2)+\tilde{h}_{3i2}\omega_J(\tilde{e}_1,\tilde{e}_3)+\tilde{h}_{4i2}\omega_J(\tilde{e}_1,\tilde{e}_4)\big]
    \\&\nonumber=-\lambda^2\big[\tilde{h}_{3i1}(-\sin\theta_1\cos\theta_2\omega_J(e_1,e_2)+\cos\theta_1\sin\theta_2\omega_J(e_3,e_4))\big]
    \\&\kern1em-\lambda^2\big[\tilde{h}_{4i2}(-\cos\theta_1\sin\theta_2\omega_J(e_1,e_2)+\sin\theta_1\cos\theta_2\omega_J(e_3,e_4))\big].
\end{align}
Thus, $|\tilde{e}_i(\lambda)|\leq cs|\mathrm{II}^t|$ and we complete the proof.

\end{proof}

\begin{proposition}\label{P4.12}
Let $(M,\overline{g},\overline{\nabla},I,J,K)$ be a hyperkähler $4$-manifold, $\Sigma\subset M$ be a compact, oriented, strongly stable complex Lagrangian surface with respect to $I,K$ ,i.e., $\omega_I+i\omega_K\equiv 0$ on $\Sigma$. Let $\Gamma\subset U_\epsilon$ be a Lagrangian surface with
\begin{align} \label{CC}
    \sup_{q\in\Gamma}\big(1-(*\Omega)+K\psi\big)<\kappa \quad \textrm{ for some constant $0<\kappa<<1$ and $K>0$.} 
\end{align}
Then, the solution of the special hyperk\"{a}hler flow $\Gamma^t$ with $\Gamma^0=\Gamma$ satisfies (\ref{CC}) if it exists.   
\end{proposition}
\begin{proof}

Let $*\Omega=\Omega(\Tilde{e}_1,\Tilde{e}_2)$. By \Cref{P2.11}, we get

\begin{align}
      \frac{\partial}{\partial t} (*\Omega)&\nonumber=\lambda^2[\Delta(*\Omega)+*\Omega(|\mathrm{II}^{t}|^2)+(\Tilde{\Omega}_{32}\overline{R}_{\tilde{1}\tilde{k}\tilde{k}\tilde{3}}+\Tilde{\Omega}_{42}\overline{R}_{\tilde{1}\tilde{k}\tilde{k}\tilde{4}}+\Tilde{\Omega}_{13}\overline{R}_{\tilde{2}\tilde{k}\tilde{k}\tilde{3}}+\Tilde{\Omega}_{14}\overline{R}_{\tilde{2}\tilde{k}\tilde{k}\tilde{4}})]
    \\&\kern1em\nonumber+2\lambda[\Tilde{e}_1(\lambda)(\Tilde{\Omega}_{32}\Tilde{H}_3+\Tilde{\Omega}_{42}\Tilde{H}_4)+\Tilde{e}_2(\lambda)(\Tilde{\Omega}_{13}\Tilde{H}_3+\Tilde{\Omega}_{14}\Tilde{H}_4)]
    \\&\kern1em\nonumber+\lambda \tilde{e}_i(\lambda)[\Tilde{h}_{31i}\Tilde{\Omega}_{32}+\Tilde{h}_{41i}\Tilde{\Omega}_{42}+\Tilde{h}_{32i}\Tilde{\Omega}_{13}+\Tilde{h}_{42i}\Tilde{\Omega}_{14}]-2\lambda^2\Tilde{\Omega}_{34}(\Tilde{h}_{3k1}\Tilde{h}_{4k2}-\Tilde{h}_{3k2}\Tilde{h}_{4k1})
    \\&\kern1em+(\overline{\nabla}_{\lambda\nabla\lambda}\Tilde{\Omega})(\Tilde{e}_1,\Tilde{e}_2)-\sum_{k=1,2}\lambda^2(\overline{\nabla}^2_{\tilde{e}_k,\tilde{e}_k}\overline{\Omega})(\tilde{e}_1,\tilde{e}_2)-2\sum_{\alpha=3,4}\lambda^2\Tilde{\Omega}_{\alpha 2,k}\Tilde{h}_{\alpha k1}-2\sum_{k=1,2}\sum_{\alpha=3,4}\lambda^2\Tilde{\Omega}_{1\alpha ,k}\Tilde{h}_{\alpha k2},  
\end{align}
where $\Tilde{\Omega}_{AB}=\Omega(\Tilde{e}_A,\Tilde{e}_B)$ and $\Tilde{h}_{\alpha ij}=\overline{g}(\Tilde{e}_\alpha,\overline{\nabla}_{\Tilde{e}_i}\Tilde{e}_j)$. Note that
\begin{align*}
  \Omega(\Tilde{e}_A,\Tilde{e}_B)=
   \begin{pmatrix}
   0 & \cos\theta_1\cos\theta_2 & 0 &-\cos\theta_1\sin\theta_2 \\
   -\cos\theta_1\cos\theta_2 & 0 & \cos\theta_2\sin\theta_1 & 0 \\
   0 & -\cos\theta_2\sin\theta_1 & 0 & \sin\theta_1\sin\theta_2 \\
   \cos\theta_1\sin\theta_2 & 0 & -\sin\theta_1\sin\theta_2 & 0
   \end{pmatrix},
\end{align*}
we derive the following by (\ref{ei}) and \Cref{L4.8}.
\begin{enumerate}
    \item [(a)]
\begin{align*}
    \Big|2\lambda[\Tilde{e}_1(\lambda)(\Tilde{\Omega}_{32}\Tilde{H}_3+\Tilde{\Omega}_{42}\Tilde{H}_4)+\Tilde{e}_2(\lambda)(\Tilde{\Omega}_{13}\Tilde{H}_3+\Tilde{\Omega}_{14}\Tilde{H}_4)]\Big|\leq c_1\lambda^2s^2|\mathrm{II}^t|^2,
\end{align*}
\item[(b)]
\begin{align*}
    \Big|\lambda \tilde{e}_i(\lambda)[\Tilde{h}_{31i}\Tilde{\Omega}_{32}+\Tilde{h}_{41i}\Tilde{\Omega}_{42}+\Tilde{h}_{32i}\Tilde{\Omega}_{13}+\Tilde{h}_{42i}\Tilde{\Omega}_{14}]\Big|\leq c_2\lambda^2s^2|\mathrm{II}^t|^2,
\end{align*}
\item [(c)]
\begin{align*}
\Big|2\lambda^2\Tilde{\Omega}_{34}(\Tilde{h}_{3k1}\Tilde{h}_{4k2}-\Tilde{h}_{3k2}\Tilde{h}_{4k1})\Big|\leq c_3\lambda^2s^2|\mathrm{II}^t|^2,
\end{align*}
\item [(d)]
\begin{align*}
    \Big|(\overline{\nabla}_{\lambda\nabla\lambda}\Tilde{\Omega})(\Tilde{e}_1,\Tilde{e}_2)\Big|\leq c_4\lambda^2(s^2+\psi)|\mathrm{II}^t|.
\end{align*}
\end{enumerate}    
Apply \Cref{L4.8} and \Cref{L4.9}, we conclude that
\begin{align}\label{so}
    \nonumber\frac{\partial}{\partial t}(*\Omega)&\nonumber  \geq \lambda^2\Big[\Delta(*\Omega)+*\Omega(|\mathrm{II}^{t}|^2)-2\langle \mathrm{II}^t,\mathrm{II}^\Sigma+S^\Sigma \rangle(*\Omega)+|\mathrm{II}^\Sigma+S^\Sigma|^2(*\Omega)-c_5(s^2+\psi)(|\mathrm{II}^t|^2+|\mathrm{II}^t|+1)\Big]
    \\&\geq \lambda^2\Big[\Delta(*\Omega)+(*\Omega)|\mathrm{II}^{t}-\mathrm{II}^{\Sigma}-S^\Sigma|^2-c_6(s^2+\psi)|\mathrm{II}^t|^2-c_6(s^2+\psi)\Big].
\end{align}
From $C^0$ estimate (\ref{c0}), we have 
\begin{align*}
    \frac{\partial}{\partial t}(\psi)\leq \lambda^2\Big[\Delta(\psi)-c_7(s^2+\psi)+c_8s(s^2+\psi)(|\mathrm{II}^t|^2+1)\Big].
\end{align*}
Then, 
\begin{align*}
    \frac{\partial}{\partial t}(*\Omega-K\psi)&\nonumber\geq \lambda^2\Big[\Delta(*\Omega-K\psi)+(*\Omega)|\mathrm{II}^{t}-\mathrm{II}^{\Sigma}-S^\Sigma|^2
     \\&\kern1em+(Kc_8s-c_6)(s^2+\psi)|\mathrm{II}^t|^2+(Kc_7-c_6-Kc_8s)(s^2+\psi)\Big],
\end{align*}
where $K$ is a constant which will be determined later. Now, we take $\epsilon>0$ such that $ \sup_{q\in\Gamma}(1-(*\Omega)+K\psi)<\kappa<\epsilon^2$ then $*\Omega>1-\epsilon^2$, $\psi<\epsilon^2$. Also, 
\begin{align*}
    \cos\theta_i\geq \cos\theta_1\cos\theta_2> 1-\epsilon^2 \Longrightarrow s\leq \sqrt{2}\epsilon.
\end{align*}
We can pick our $K$ satisfying $Ks<\sqrt{\epsilon}$ so that 
\begin{align*}
    Kc_8s-c_6<0 \quad \textrm{ and } \quad Kc_7-c_6-Kc_8s>0.
\end{align*}
Thus,
\begin{align*}
     \frac{\partial}{\partial t}(*\Omega-K\psi)&\nonumber\geq \lambda^2\Big[\Delta(*\Omega-K\psi)+(*\Omega)(|\mathrm{II}^{t}-\mathrm{II}^{\Sigma}|^2-|S^\Sigma|^2)
     \\&\kern1em+(Kc_8s-c_6)(s^2+\psi)(|\mathrm{II}^t-\mathrm{II}^{\Sigma}|^2+|\mathrm{II}^{\Sigma}|^2)+(Kc_7-c_6-c_8Ks)(s^2+\psi)\Big].
\end{align*}
Note that $|S^\Sigma|^2\leq c_9\psi$ and $|\mathrm{II}^{\Sigma}|^2\leq c_{10}$. If we pick $\epsilon$ small enough, we derive 
\begin{align*}
    \frac{\partial}{\partial t}(*\Omega-K\psi)&\nonumber\geq \lambda^2\Big[\Delta(*\Omega-K\psi)+\big( *\Omega+(Kc_8s-c_6)(s^2+\psi)\big)|\mathrm{II}^{t}-\mathrm{II}^{\Sigma}|^2
     \\&+\big(Kc_7-c_6-Kc_8s-c_9+c_{10}(Kc_8s-c_6)\big)(s^2+\psi)\Big].
\end{align*}
Finally, we choose $K$ large enough so that  
\begin{align*}
    \frac{\partial}{\partial t}(*\Omega-K\psi)\geq \lambda^2\Big[\Delta(*\Omega-K\psi)+\frac{1}{2}\big((*\Omega)-K\psi\big)|\mathrm{II}^{t}-\mathrm{II}^{\Sigma}|^2\Big]+c_{11}(s^2+K\psi).
\end{align*}
We remark that 
\begin{align*}
    1-*\Omega\leq 1-(*\Omega)^2=\sin^2\theta_1+\sin^2\theta_2-\sin^2\theta_1\sin^2\theta_2\leq 2s^2.
\end{align*}
Thus,
\begin{align*}
    \frac{\partial}{\partial t}(1-*\Omega+K\psi)\leq \lambda^2\Big[\Delta(1-*\Omega+K\psi)-\frac{1}{2}((*\Omega)-K\psi)|\mathrm{II}^{t}-\mathrm{II}^{\Sigma}|^2\Big]-c_{12}(1-*\Omega+K\psi).
\end{align*}
By the maximum principle, $1-*\Omega-K\psi$ is non-decreasing. We then complete the proof.
\end{proof}

\subsection{Long time existence}
\begin{lemma}
The special $H$-flow $\Gamma^t$ with $\Gamma^0=\Gamma$ exists all time.
\end{lemma}
\begin{proof}
Recall that in the special $H$-flow, we have (\ref{A}) 
\begin{align*}
    \frac{\partial}{\partial t}|\mathrm{II}^t|^2\leq \lambda^2\Big[\Delta|\mathrm{II}^t|^2+c(|\mathrm{II}^t|^4+1)\Big].
\end{align*}    
We try to use the maximum principle to prove that $|\mathrm{II}^t|$ is uniformly bounded. In order to do that we consider a constant $p>1$, then it follows from (\ref{so}) that
\begin{align*}
    \frac{\partial}{\partial t}(*\Omega)^p&\nonumber =p(*\Omega)^{p-1}\frac{\partial}{\partial t}(*\Omega)
    \\&\geq  p(*\Omega)^{p-1} \lambda^2\Big[\Delta(*\Omega)+(*\Omega)|\mathrm{II}^{t}-\mathrm{II}^{\Sigma}-S^\Sigma|^2-c_1(s^2+\psi)|\mathrm{II}^t|^2-c_1(s^2+\psi)\Big]
    \\&\geq \lambda^2\Big[\Delta(*\Omega)^p-p(p-1)(*\Omega)^{p-2}|\nabla(*\Omega)|^2+p(*\Omega)^p|\mathrm{II}^{t}-\mathrm{II}^{\Sigma}-S^\Sigma|^2-c_1p(s^2+\psi)|\mathrm{II}^t|^2-c_1p(s^2+\psi)\Big].
\end{align*}
Here, we note that $\Delta(*\Omega)^p=p(*\Omega)^{p-1}\Delta(*\Omega)+p(p-1)(*\Omega)^{p-2}|\nabla(*\Omega)|^2$. Using \Cref{L4.8}, we derive that
\begin{align*}
    \frac{\partial}{\partial t}(*\Omega)^p&\geq \lambda^2\Big[\Delta(*\Omega)^p+p(*\Omega)^p|\mathrm{II}^{t}-\mathrm{II}^{\Sigma}-S^\Sigma|^2-c_2p^2s^2 |\mathrm{II}^\Gamma-\mathrm{II}^\Sigma|^2-c_1p(s^2+\psi)|\mathrm{II}^t|^2-c_3p^2(s^2+\psi)\Big].
\end{align*}
Again, using (\ref{c0}),
\begin{align*}
    \frac{\partial}{\partial t}(\psi)\leq \lambda^2\Big[\Delta(\psi)+c_4s(s^2+\psi)|\mathrm{II}^t|^2-c_5(s^2+\psi)\Big] ,
\end{align*}
we compute that
\begin{align*}
    \frac{\partial}{\partial t}((*\Omega)^p-K'\psi)&\nonumber\geq \lambda^2\Big[\Delta((*\Omega)^p-K'\psi)+p(*\Omega)^p|\mathrm{II}^{t}-\mathrm{II}^{\Sigma}-S^\Sigma|^2-c_2p^2s^2 |\mathrm{II}^\Gamma-\mathrm{II}^\Sigma|^2
    \\&\kern1em\nonumber-(c_1p+c_4K's)(s^2+\psi)|\mathrm{II}^t|^2+(c_5K'-c_3p^2)(s^2+\psi)\Big]
    \\&\nonumber\geq \lambda^2\Big[\Delta((*\Omega)^p-K'\psi)+(p(*\Omega)^p-c_2p^2s^2-(c_1p+c_4K's)(s^2+\psi))|\mathrm{II}^t-\mathrm{II}^\Sigma|^2
    \\&\kern1em+(c_5K'-c_3p^2-c_6p-c_1c_7p-c_4c_7K's)(s^2+\psi)\Big],
\end{align*}
where $|S^\Sigma|^2\leq c_6\psi$ and $|\mathrm{II}^\Sigma|\leq c_7$. We choose $p,K'$ and $s$ so that 
\begin{align}\label{OP}
     \frac{\partial}{\partial t}((*\Omega)^p-K'\psi)\geq \lambda^2\Big[\Delta ((*\Omega)^p-K'\psi)+\frac{p}{3}((*\Omega)^p-K'\psi)|\mathrm{II}^{t}-\mathrm{II}^{\Sigma}|^2\Big].
\end{align}
 Take $\epsilon$ small enough so that $((*\Omega)^p-K'\psi)>\frac{1}{2}$ at $\Gamma$ then it is non-decreasing by maximum principle.

Define $\eta=\big((*\Omega)^p-K'\psi\big)$, then 
\begin{align*}
     \frac{\partial}{\partial t}(\eta^{-1}|\mathrm{II}^t|^2)&\nonumber=\eta^{-1}\frac{\partial}{\partial t}(|\mathrm{II}^t|^2)-\eta^{-2}\frac{\partial}{\partial t}(\eta)|\mathrm{II}^t|^2
     \\&\leq \lambda^2\Big[\eta^{-1}\big(\Delta|\mathrm{II}^t|^2+c_8|\mathrm{II}^t|^4+c_8\big)-\eta^{-2}|\mathrm{II}^t|^2\big(\Delta (\eta)+\frac{p}{3}\eta|\mathrm{II}^{t}-\mathrm{II}^{\Sigma}|^2\big)\Big].
\end{align*}
Note that 
\begin{align*}
    \Delta(\eta^{-1}|\mathrm{II}^t|^2)&\nonumber=\Delta(\eta^{-1})|\mathrm{II}^t|^2+\eta^{-1}\Delta(|\mathrm{II}^t|^2)+2\big\langle \nabla(\eta^{-1}),\nabla(|\mathrm{II}^t|^2) \big\rangle
    \\&=-\eta^{-2}|\mathrm{II}^t|^2\Delta\eta+\eta^{-1}\Delta(|\mathrm{II}^t|^2)-2\eta^{-1}\big\langle \nabla(\eta),\nabla(\eta^{-1}|\mathrm{II}^t|^2) \big\rangle,
\end{align*}
so
\begin{align*}
     \frac{\partial}{\partial t}(\eta^{-1}|\mathrm{II}^t|^2)\leq \lambda^2\Big[ \Delta(\eta^{-1}|\mathrm{II}^t|^2)+2\eta^{-1}\big\langle \nabla(\eta),\nabla(\eta^{-1}|\mathrm{II}^t|^2) \big\rangle+c_8\eta^{-1}(|\mathrm{II}^t|^4+1)-\eta^{-1}\frac{p}{3}|\mathrm{II}^{t}-\mathrm{II}^{\Sigma}|^2|\mathrm{II}^t|^2\Big]
\end{align*}
Finally, since
\begin{align*}
    |\mathrm{II}^{t}-\mathrm{II}^{\Sigma}|^2\geq |\mathrm{II}^t|^2-|\mathrm{II}^{\Sigma}|^2\geq |\mathrm{II}^t|^2-c_9,
\end{align*}
we can choose $p$ large enough so that $\frac{p}{3}>c_7$ then
\begin{align*}
    \frac{\partial}{\partial t}(\eta^{-1}|\mathrm{II}^t|^2)\leq \lambda^2\Big[ \Delta(\eta^{-1}|\mathrm{II}^t|^2)+2\eta^{-1}\langle \nabla(\eta),\nabla(\eta^{-1}|\mathrm{II}^t|^2) \rangle+(c_8-\frac{p}{3})\eta^{-1}|\mathrm{II}^t|^4+c_9\eta^{-1}\frac{p}{3}|\mathrm{II}^{t}|^2+c_8\eta^{-1}\Big].
\end{align*}
By the maximum principle, $\eta^{-1}|\mathrm{II}^t|^2$ is uniformly bounded, so is $|\mathrm{II}^t|^2$.
\end{proof}

\subsection{Proof of Main Result}
Now, we are ready to prove the main result. 

\begin{theorem}
Let $(M,\overline{g},\overline{\nabla},I,J,K)$ be a $\textrm{hyperk}\Ddot{a}\textrm{hler}$ $4$-manifold, $\Sigma\subset M$ be a compact, oriented, strongly stable complex Lagrangian surface with respect to $I,K$ ,i.e., $\omega_I+i\omega_K\equiv 0$ on $\Sigma$ and $\Gamma\subset M$ be a Lagranigan surface with respect to $K$ which is $C^1$ close to $\Sigma$. Consider 
\begin{align*}
    \mathcal{N}=\{f:\Gamma\to M| f^\star(\omega_J)=\rho, f^\star(\omega_K)=0 \textrm{ where $\rho$ is a given volume form on $\Gamma$}\}
\end{align*}
the special H-flow $f_t(\Gamma)=\Gamma^t\in \mathcal{N}$ with $f_0(\Gamma)=\Gamma$ exists for all time and converge to $\Sigma$ smoothly.
\end{theorem}

\begin{proof}
 Due to long time existence, we know that the special $H$-flow exists for all time. From $C^0$ estimate (\Cref{P4.11}), we see that 
 \begin{align*}
\frac{\partial}{\partial t}(\psi)\leq \lambda^2(\Delta^{\Gamma_t}(\psi)+\langle \nabla(\frac{\lambda^2}{2}),\nabla\psi \rangle-c_1\psi. 
\end{align*}    
Thus, $\psi(x,t)\leq \sup\psi|_{t=0} e^{-c_1 t}$ ,i.e., it converges to 0 exponentially. Also, from $C^1$ estimate (\Cref{P4.12}), we have
\begin{align*}
    \frac{\partial}{\partial t}(1-*\Omega+K\psi)\leq \lambda^2\Big[\Delta(1-*\Omega+K\psi)-\frac{1}{2}((*\Omega)-K\psi)|\mathrm{II}^{t}-\mathrm{II}^{\Sigma}|^2\Big]-c_{2}(1-*\Omega+K\psi),
\end{align*}
then $(1-*\Omega+K\psi)(x,t)\leq \sup(1-*\Omega+K\psi)|_{t=0} e^{-c_2 t}$ ,i.e., $(*\Omega)$ converge to 1 exponentially by $C^1$ estimate.

For $C^2$ convergence, we need to use the following
\begin{align} \label{KE}
    &\nonumber \frac{\partial}{\partial t}\eta \geq \lambda^2[\Delta\eta+\frac{p}{3}\eta |\mathrm{II}^{t}-\mathrm{II}^{\Sigma}|^2 ],
    \\& \frac{\partial}{\partial t}|\mathrm{II}^{t}-\mathrm{II}^{\Sigma}|^2\leq \lambda^2\Big[\Delta|\mathrm{II}^{t}-\mathrm{II}^{\Sigma}|^2+c(|\mathrm{II}^{t}-\mathrm{II}^{\Sigma}|^4+1)\Big],
\end{align}
where the first inequality is (\ref{OP}) and the second one is followed by (\ref{A}). 

\textbf{Claim 1:} $    \int_0^\infty (\int_{\Gamma_t}|\mathrm{II}^{t}-\mathrm{II}^{\Sigma}|^2 d\mu_t)dt\leq C_1$ for some positive constant $C_1$.
\\
\\First see that 
\begin{align*}
    \frac{\partial}{\partial t}\textrm{Vol}(\Gamma_t)=\int_{\Gamma_t}\frac{\partial}{\partial t}d\mu_t=\int_{\Gamma_t} \left( \Delta ( \frac{\lambda^2}{2}
) - \lambda^2 |H|^2 \right) d\mu_t=\int_{\Gamma_t} - \lambda^2 |H|^2  d\mu_t\leq 0.
\end{align*}
Thus, $\textrm{Vol}(\Gamma_t)$ is non-increasing and has a lower bound. Then, the limit exists. Next, we integrate (\ref{KE}) and note that $\eta$ has a lower bound and converges to 1 uniformly.
\begin{align*}
    \int_{\Gamma_t}\frac{\partial}{\partial t}\eta d\mu_t\geq  \int_{\Gamma_t} \lambda^2\Delta\eta d\mu_t+ c_3 \int_{\Gamma_t}|\mathrm{II}^{t}-\mathrm{II}^{\Sigma}|^2 d\mu_t.
\end{align*}
Rewrite the left hand side  
\begin{align*}
\int_{\Gamma_t}\frac{\partial}{\partial t}\eta d\mu_t&= \frac{\partial}{\partial t}  \int_{\Gamma_t}\eta d\mu_t-\int_{\Gamma_t} \eta(\frac{\partial}{\partial t} d\mu_t)
\\&=\frac{\partial}{\partial t}  \int_{\Gamma_t}\eta d\mu_t-\int_{\Gamma_t} \eta \left( \Delta ( \frac{\lambda^2}{2}
) - \lambda^2 |H|^2 \right) d\mu_t
\\&=\frac{\partial}{\partial t}  \int_{\Gamma_t}\eta d\mu_t+\int_{\Gamma_t}\frac{1}{2}\langle \nabla(\lambda^2),\nabla\eta\rangle d\mu_t+\int_{\Gamma_t}\eta\lambda^2|H|^2 d\mu_t,
\end{align*}
where we use the Stoke's theorem. Thus
\begin{align}\label{2T}
   \frac{\partial}{\partial t}  \int_{\Gamma_t}\eta d\mu_t+\int_{\Gamma_t}\frac{3}{2}\langle \nabla(\lambda^2),\nabla\eta\rangle d\mu_t+\int_{\Gamma_t}\eta\lambda^2|H|^2 d\mu_t\geq  c_3 \int_{\Gamma_t}|\mathrm{II}^{t}-\mathrm{II}^{\Sigma}|^2 d\mu_t.
\end{align}
Recall that $|\nabla(\lambda^2)|$ is uniformly bounded (since $|\mathrm{II}^{t}|$ is uniformly bounded) and $|\nabla\eta|^2\leq c_4(|\nabla(*\Omega)|^2+|\nabla\psi|^2)$. It can be seen from \Cref{L4.8} and the proof in $C^0$ estimate such that  
\begin{align*}
   |\nabla\eta|^2\leq c_5(s^2+\psi)\leq c_6(1-*\Omega+K\psi)\leq c_7e^{-c_8 t}.
\end{align*}
Now, we can finish the proof of claim 1. The first term is
\begin{align*}
    \int_0^t \frac{\partial}{\partial s}  \int_{\Gamma_s}\eta d\mu_s ds=\int_{\Gamma_t}\eta d\mu_t-\int_{\Gamma_0}\eta d\mu_0.
\end{align*}
Since $\eta$ converges to 1 uniformly, it converges. The second term is 
\begin{align*}
   \int_0^t \int_{\Gamma_s}\frac{3}{2}\langle \nabla(\lambda^2),\nabla\eta\rangle d\mu_s ds\leq \int_0^t \int_{\Gamma_s}\frac{3}{2}(|\nabla(\lambda^2)||\nabla\eta|) d\mu_s ds\leq \int_0^t c_7e^{-c_8 s}\textrm{Vol}(\Gamma_0) ds<\infty.
\end{align*}
The last term is 
\begin{align*}
   \int_0^t \int_{\Gamma_s}\eta\lambda^2|H|^2 d\mu_s ds \leq \int_0^t \int_{\Gamma_s}\lambda^2|H|^2 d\mu_s ds=\int_0^t -\frac{\partial}{\partial s}\int_{\Gamma_s}d\mu_s ds=\textrm{Vol}(\Gamma_0)-\textrm{Vol}(\Gamma_t).
\end{align*}
The proof ofr claim 1 is followed by the inequality (\ref{2T}).

\textbf{Claim 2:} $\frac{d}{dt}\int_{\Gamma_t}|\mathrm{II}^{t}-\mathrm{II}^{\Sigma}|^2 d\mu_t\leq C_2$ for some positive constant $C_2$.
\\By (\ref{KE}), we get
\begin{align*}
    \frac{\partial}{\partial t}|\mathrm{II}^{t}-\mathrm{II}^{\Sigma}|^2\leq \lambda^2\Delta|\mathrm{II}^{t}-\mathrm{II}^{\Sigma}|^2+c_9.
\end{align*}
Since $|\mathrm{II}^{t}|$ is uniformly bounded, we compute 
\begin{align*}
    \frac{d}{dt}\int_{\Gamma_t}|\mathrm{II}^{t}-\mathrm{II}^{\Sigma}|^2 d\mu_t&=\int_{\Gamma_t}\frac{\partial}{\partial t}|\mathrm{II}^{t}-\mathrm{II}^{\Sigma}|^2 d\mu_t+\int_{\Gamma_t}|\mathrm{II}^{t}-\mathrm{II}^{\Sigma}|^2(\frac{\partial}{\partial t} d\mu_t)
    \\&\leq \int_{\Gamma_t}\Big(\lambda^2\Delta|\mathrm{II}^{t}-\mathrm{II}^{\Sigma}|^2+\Delta(\frac{\lambda^2}{2})|\mathrm{II}^{t}-\mathrm{II}^{\Sigma}|^2- \lambda^2 |H|^2|\mathrm{II}^{t}-\mathrm{II}^{\Sigma}|^2+c_9\Big)d\mu_t
    \\&\leq \int_{\Gamma_t} \Big( -\frac{3}{2}\langle \nabla(\lambda^2),\nabla|\mathrm{II}^{t}-\mathrm{II}^{\Sigma}|^2\rangle +c_9\Big)d\mu_t<\infty.
\end{align*}
because $\big|\nabla|\mathrm{II}^{t}-\mathrm{II}^{\Sigma}|\big|^2$ is unifromly bounded.

Now, by claim 1 and claim 2 we can use Lemma 6.3 in \cite{WT2} to derive that
\begin{align*}
    \int_{\Gamma_t}|\mathrm{II}^{t}-\mathrm{II}^{\Sigma}|^2 d\mu_t \to 0.
\end{align*}
Then, it is a standard argument to show that $|\mathrm{II}^{t}-\mathrm{II}^{\Sigma}|\longrightarrow 0$ and $\Gamma^t$ converges to $\Sigma$ smoothly.

\end{proof}

\begin{corollary}
Let $\Gamma$ be a Lagrangian surface with respect to $K$ in the Eguchi-Hanson space, which is $C^1$ close to the zero section $S^2$. Then, the special $H$-flow $\Gamma^t$ with $\Gamma^0=\Gamma$ exists for all time and converges to $S^2$ smoothly.
\end{corollary}

\bibliographystyle{plain}
\bibliography{Reference}
\nocite{*}

\end{document}